\documentclass[11pt, twoside]{amsart}
\usepackage{amssymb,amsmath,amsfonts,amsthm}
\usepackage{mathrsfs}
\usepackage{bbm}
\usepackage[normalem]{ulem}
\usepackage[colorlinks, citecolor=blue,pagebackref,hypertexnames=false]{hyperref}

\usepackage{tikz}
\usetikzlibrary{arrows.meta,positioning,decorations.pathreplacing}
\usepackage{xcolor}
\usepackage{pgfplots}
\usetikzlibrary{3d, calc}
\pgfplotsset{compat=1.18}

\allowdisplaybreaks
\theoremstyle{plain}
\newtheorem{thm}{Theorem}[section]
\newtheorem{prop}[thm]{Proposition}
\newtheorem{lem}[thm]{Lemma}
\newtheorem{cor}[thm]{Corollary}

\theoremstyle{definition}
\newtheorem{defn}[thm]{Definition}
\newtheorem{remark}[thm]{Remark}

\begin{document}

\title[Haar bases for multi-parameter twisted structures]{\bf\Large  Haar bases for multi-parameter twisted structures}
 \author{Ji Li}
\address{Ji Li, School of Mathematical and Physical Sciences, Macquarie University, NSW 2109, Australia}
\email{ji.li@mq.edu.au}

\author{Chong-Wei Liang}
\address{Chong-Wei Liang, Department of Mathematics, National Taiwan University, Taiwan}
\email{d10221001@ntu.edu,tw}

\author{Brett D. Wick}
\address{Brett D. Wick, Department of Mathematics, Washington University - St. Louis, St. Louis, MO 63130-4899 USA}
\email{wick@math.wustl.edu}

\author{Liangchuan Wu}
\address{Liangchuan Wu, School of Mathematical Science, Anhui University, Hefei, 230601,  P.R.~China}
\email{wuliangchuan@ahu.edu.cn}

\author{Qingyan Wu}
\address{Qingyan Wu, School of Mathematics and Statistics, Linyi University, Linyi 276000,   P.R.~China}
\email{wuqingyan@lyu.edu.cn}
\date{\today}

\begin{abstract}
Motivated by the Cauchy--Szeg\H{o} projections on a broad class of Siegel domains and the geometric quotient structures of nilpotent Lie groups observed by Nagel, Ricci, and Stein, we develop a martingale and Haar wavelet framework for twisted multi-parameter geometries. We introduce twisted dyadic filtrations and construct adapted Haar bases on Euclidean spaces $\mathbb{R}^{2m}$. 
Each of the resulting dyadic systems forms a complete orthonormal basis of $L^2(\mathbb R^{2m})$, and their union yields a tight frame with frame bound $3$.
We establish $L^p$-equivalences for the associated discrete twisted Littlewood--Paley square functions. 

Furthermore, we extend this discrete real-variable theory to the non-abelian setting of a nilpotent Lie group of step two, $\mathscr{N}$, which serves as the Shilov boundary of certain fundamental Siegel domains. By projecting product fractal tiles from a lifting group of Heisenberg products, we define twisted dyadic shards and construct twisted nilpotent Haar  frames.  More precisely, we first introduce raw projected shards that reflect the quotient geometry, and then pass to analytic dyadic shards which are exactly rectifiable and remain uniformly comparable to the raw quotient structure in the relevant scale regimes.  
This yields a discrete framework adapted to twisted quotient geometries in both the Euclidean and nilpotent settings, providing a basic dyadic infrastructure for further developments in twisted real-variable theory.

\end{abstract}

\maketitle

\section{Introduction and Definitions}
\setcounter{equation}{0}

The development of multi-parameter harmonic analysis has historically been driven by the geometric complexities of boundary value problems in several complex variables. While classical one-parameter Calder\'on-Zygmund theory provides the necessary tools for strictly pseudoconvex domains, domains with product structures require multi-parameter frameworks. The classical product theory on $\mathbb{R}^n$, pioneered by S.-Y. A. Chang, R. Fefferman, Gundy, Journ\'e, Pipher, and Stein \cite{CF80, CF85, Feff-S, J, P, St98}, successfully characterize operators with independent scaling parameters. Subsequently, the study of the $\bar{\partial}$-Neumann problem on the Heisenberg group led to the discovery of multiparameter flag structures, introduced by M\"uller, Ricci, and Stein \cite{MRS} and extensively developed in recent decades \cite{CCLLO, HLLW, NRSW1, NRSW2}.

Despite these successes, a significant geometric and analytic challenge remains when passing to quotient spaces. As acutely observed by Nagel, Ricci, and Stein \cite{NRS}, the class of standard product operators is not closed under passage to a quotient subgroup. A quintessential example is the Cauchy--Szeg\H{o} projection on a wide class of quadratic surfaces of higher codimension in $\mathbb{C}^n$. The underlying geometry is naturally modeled by appropriate quotients of products of Heisenberg groups. However, projecting a well-behaved product kernel down to the quotient space couples the independent singularities, creating a highly complex, ``twisted'' singular integral kernel that is singular along intersecting hyperplanes. These operators explicitly fall outside the scope of established product and flag theories.

In recent related works \cite{FLLWW,LLWW}, continuous real-variable methods for these twisted multiparameter singular integrals have been initiated, including the development of continuous tube maximal functions, continuous Littlewood--Paley area functions, and atomic decompositions for twisted Hardy spaces on $\mathbb{R}^{2m}$. However, a discrete framework, specifically a dyadic martingale theory and an adapted Haar wavelet basis, has not yet been provided.  In more classical settings, dyadic Haar systems and martingale structures are by now well understood on spaces of homogeneous type and on product spaces thereof; see, for instance, \cite{KLPW}. 

Discrete dyadic structures are highly desirable in harmonic analysis. By constructing exact orthonormal bases   for the individual dyadic systems, and hence tight frames from their unions,  we can isolate orthogonal oscillations at every scale and location, effectively bypassing the technical complexities inherent in continuous singular integral bounds.


 The present work is not merely a dyadic counterpart of the continuous twisted theory. In the classical one-parameter and product settings, dyadic filtrations, martingale differences, and Haar expansions play a basic role in square function estimates, $H^1$--BMO theory, paraproducts, distribution inequalities, and endpoint arguments; see, for example, \cite{KLPW} and the references therein. In the twisted quotient setting, however, such a dyadic structure is not available from the outset, because the quotient procedure couples the originally independent singular directions and breaks the standard product geometry. The constructions in this paper provide an exact dyadic model for this setting: they recover orthogonal and martingale structures at the discrete level, while remaining comparable to the underlying quotient tube geometry. In this way, the twisted Haar systems and dyadic shards developed here furnish a natural discrete framework for twisted multi-parameter quotient spaces.

The primary purpose of this paper is to construct these twisted dyadic multi-parameter filtrations and their corresponding Haar bases, bridging the gap between the discrete martingale theory and the continuous twisted quotient structures. We accomplish this in two major settings:
\smallskip

{(1). The Euclidean twisted structure, $\mathbb{R}^{2m}$:}\\ We analyze the geometric structure induced by the quotient map $\pi: \mathbb{R}^{3m} \to \mathbb{R}^{2m}$, defined by $$\pi(x_1, x_2, x_3) = (x_1+x_3, x_2+x_3).$$
To capture the coupling of the parameters, we classify the resulting dyadic tubes into distinct geometric systems generated by  the volume-preserving affine shears ($T_2$ and $T_3$). By pulling back the standard tensor-product Haar basis through these transformations, we construct complete orthonormal \textit{twisted Haar bases} that encode the directional cancellations along the coordinate axes and the diagonal.  We prove that the union of these systems forms a tight frame for $L^2(\mathbb{R}^{2m})$ with frame bound $3$
 and establish the twisted Burkholder--Davis--Gundy inequalities and dyadic Littlewood--Paley $L^p$-equivalences through the corresponding martingale projections.

\smallskip

{(2). The nilpotent twisted structure, $\mathscr{N}$:}\\  We subsequently lift this discrete framework to the non-abelian setting.\\Consider the nilpotent Lie group $\mathscr{N}  \cong\mathbb{C}^{n_1+n_2+n_3} \times \mathbb{R}^2$ of step two, which acts as the quotient of the lifting group $\tilde{\mathscr{N}} = \mathscr{H}_1 \times \mathscr{H}_2 \times \mathscr{H}_3$, the product of three Heisenberg groups. Because standard dyadic grids do not exist on Heisenberg groups, we rely on fractal tilings \cite{St, Ty}. By projecting the product fractal tiles from $\tilde{\mathscr{N}}$ down to $\mathscr{N}$, we define \textit{twisted dyadic shards}.  More precisely, we first define the raw projected shards arising from the quotient of products of stacked Heisenberg tiles, and then introduce analytic dyadic shards obtained from exact product structures or their inverse images under measure-preserving shears. We then construct the twisted nilpotent filtrations and unitary pullbacks to establish the twisted nilpotent Haar tight frames, proving that the discrete dyadic square functions on the quotient group $\mathscr{N}$ characterize $L^p(\mathscr{N})$ for all $1 < p < \infty$.

\medskip

The paper is organized as follows.  In Section 2, we introduce the geometric affine shears on $\mathbb{R}^{2m}$, construct the twisted Haar bases, and prove the corresponding tight frame representation with frame bound $3$. In Section 3, we define the twisted multi-parameter filtrations and prove the discrete martingale properties, including the $L^p$-boundedness of the twisted martingale square functions.  Section 4 establishes the twisted dyadic Littlewood--Paley equivalences utilizing the unitary intertwining properties. Finally, in Sections 5 and 6, we transition to the nilpotent Lie group $\mathscr{N}$, introducing the twisted dyadic shards, the non-abelian twisted filtrations, and rigorously establishing the twisted nilpotent Haar tight frames and square function equivalences.

\section{Haar basis on twisted dyadic tubes and tight frame representation}
\setcounter{equation}{0}

We first recall the underlying geometric structure induced by the quotient map  considered in \cite{FLLWW}. Let $\pi: \mathbb{R}^m \times \mathbb{R}^m \times \mathbb{R}^m \to \mathbb{R}^m \times \mathbb{R}^m$ be the projection defined by $$\pi(x_1, x_2, x_3) = (x_1+x_3, x_2+x_3).$$

 In the continuous twisted theory, the quotient geometry is described in terms of five types of dyadic tubes. For the discrete Haar theory developed here, it is convenient to regroup these geometric families into three  dyadic systems:
$$
I\times J,\qquad I\times_t J,\qquad I\,{}_t\!\times J,
$$
where
$$
I\,{\times_t}\, J:=\{(x_1,x_2)\in \mathbb R^{2m}:x_1-x_2\in I,\ x_2\in J\},
$$
and
$$
I  \,{{}_t \times}\, J:=\{(x_1,x_2)\in \mathbb R^{2m}:x_1\in I,\ x_2-x_1\in J\}.
$$
These systems are rectified by the linear measure-preserving shears
\begin{equation}\label{affine}
T_2(x_1,x_2):=(x_1-x_2,x_2) \quad\text{and}\quad 
T_3(x_1,x_2):=(x_1,x_2-x_1).
\end{equation}
The two identities
$$
T_2^{-1}(I\times J)=I{\times_t} \,J\quad\text{and}\quad
T_3^{-1}(I\times J)=I\,{{}_t \times} \,J
$$
show that the slanted dyadic systems are exact pullbacks of the standard product dyadic system.

\begin{figure}[htbp]
    \centering
    \begin{tikzpicture}[
        scale=1.1,
        gridline/.style={gray, very thin},
        axis/.style={->, thick, black},
        basetile/.style={fill=blue, opacity=0.3, draw=blue, thick},
        transtile/.style={fill=red, opacity=0.3, draw=red, thick}
    ]
    \begin{scope}[shift={(0,0)}]
        \begin{scope}
            \clip (-2.5, -2.5) rectangle (2.5, 2.5);
            
            \draw[step=1cm, gridline] (-3,-3) grid (3,3);
            
            \filldraw[basetile] (0,0) rectangle (1,1);
            
            \filldraw[transtile] (-1,-2) rectangle (0,-1);
        \end{scope}
        
        \draw[axis] (-2.7, 0) -- (2.7, 0) node[right] {$x_1$};
        \draw[axis] (0, -2.7) -- (0, 2.7) node[above] {$x_2$};
        
        \node[below, align=center] at (0, -3) {\textbf{(a) Type {\rm I}}: standard grid};
    \end{scope}
    \end{tikzpicture}
    \label{fig:tiling_sys1}
\end{figure}

\begin{figure}[htbp]
    \centering
    \begin{tikzpicture}[
        scale=1.1,
        gridline/.style={gray, very thin},
        axis/.style={->, thick, black},
        basetile/.style={fill=blue, opacity=0.3, draw=blue, thick},
        transtile/.style={fill=red, opacity=0.3, draw=red, thick}
    ]
    \begin{scope}[shift={(0,0)}]
        \begin{scope}
            \clip (-2.5, -2.5) rectangle (2.5, 2.5);
            
            \foreach \y in {-3,-2,-1,0,1,2,3} {
                \draw[gridline] (-4, \y) -- (4, \y);
            }
            \foreach \m in {-6,...,6} {
                \draw[gridline] (\m-3, -3) -- (\m+3, 3);
            }
            
            \filldraw[basetile] (0,0) -- (1,0) -- (2,1) -- (1,1) -- cycle;
            
            \filldraw[transtile] (-2,-1) -- (-1,-1) -- (0,0) -- (-1,0) -- cycle;
        \end{scope}
        
        \draw[axis] (-2.7, 0) -- (2.7, 0) node[right] {$x_1$};
        \draw[axis] (0, -2.7) -- (0, 2.7) node[above] {$x_2$};
        
        \node[below, align=center] at (0, -3) {\textbf{(b) Type {\rm II} and {\rm  III}}: slanted rectangles with the base on first direction};
    \end{scope}
    \end{tikzpicture}
    \label{fig:tiling_sys2}
\end{figure}

%

To discretize the real-variable theory and build a bridge to the atomic decomposition on the twisted geometry, we now construct the Haar basis adapted to the dyadic tubes. As observed in the structure of the tubes, the singular interactions partition the geometric objects into three fundamental families, which further sub-divide into the five types of rectangles (type ${\rm I}$ through type ${\rm V}$).

\begin{figure}[htbp]
    \centering
    \begin{tikzpicture}[
        scale=1.1,
        gridline/.style={gray, very thin},
        axis/.style={->, thick, black},
        basetile/.style={fill=blue, opacity=0.3, draw=blue, thick},
        transtile/.style={fill=red, opacity=0.3, draw=red, thick}
    ]
    \begin{scope}[shift={(0,0)}]
        \begin{scope}
            \clip (-2.5, -2.5) rectangle (2.5, 2.5);
            
            \foreach \x in {-3,-2,-1,0,1,2,3} {
                \draw[gridline] (\x, -4) -- (\x, 4);
            }
            \foreach \m in {-6,...,6} {
                \draw[gridline] (-3, \m-3) -- (3, \m+3); 
            }
            
            \filldraw[basetile] (0,0) -- (1,1) -- (1,2) -- (0,1) -- cycle;
            
            \filldraw[transtile] (-2,-1) -- (-1,0) -- (-1,1) -- (-2,0) -- cycle;
        \end{scope}
        
        \draw[axis] (-2.7, 0) -- (2.7, 0) node[right] {$x_1$};
        \draw[axis] (0, -2.7) -- (0, 2.7) node[above] {$x_2$};
        
        \node[below, align=center] at (0, -3) {\textbf{(c) Type {\rm IV} and {\rm V}}: slanted rectangles with base on the second direction.};
    \end{scope}
    \end{tikzpicture}
    \label{fig:tiling_sys3}
\end{figure}   

Let $\mathscr{D}$ denote the collection of all standard dyadic cubes in $\mathbb{R}^m$. For each $I \in \mathscr{D}$, let $\{ h_I^\epsilon \}_{\epsilon \in \mathcal{E}}$ be the standard $L^2$-normalized Haar wavelet basis supported on $I$, where the index set is $\mathcal{E} := \{1, 2, \dots, 2^m-1\}$. For a standard dyadic rectangle $R = I \times J \in \mathscr{D} \times \mathscr{D}$ and a multi-index $\vec{\epsilon} = (\epsilon_1, \epsilon_2) \in \mathcal{E} \times \mathcal{E}$, we define the standard tensor product Haar function on $\mathbb{R}^{2m}$ as
\begin{align*}
    h_{I \times J}^{\vec{\epsilon}}(\mathbf{x}_1, \mathbf{x}_2) := h_I^{\epsilon_1}(\mathbf{x}_1) h_J^{\epsilon_2}(\mathbf{x}_2).
\end{align*}
It is a classical fact that the collection $$\mathcal{B}_1 := \{ h_{I \times J}^{\vec{\epsilon}} \}_{I,J \in \mathscr{D}, \, \vec{\epsilon}}$$ 
forms a complete orthonormal basis for $L^2(\mathbb{R}^{2m})$. Note that the support of $h_{I \times J}^{\vec{\epsilon}}$  matches the type ${\rm I}$ standard rectangles.

 To construct the Haar functions adapted to slanted tubes, we now use the affine shears introduced above. Since $T_2$ and $T_3$ are measure-preserving linear bijections on $\mathbb{R}^{2m}$, their pullbacks transport the standard tensor-product Haar system to the two slanted dyadic systems.
The Jacobian of $T_2$ and $T_3$ correspond to the block matrices
$$
\begin{pmatrix} I_m & -I_m \\ 0 & I_m \end{pmatrix}
\qquad\text{and}\qquad
\begin{pmatrix} I_m & 0 \\ -I_m & I_m \end{pmatrix},
$$
respectively. Both matrices have determinant $1$, making them measure-preserving on $\mathbb{R}^{2m}$.

We define twisted Haar functions by pulling back the standard Haar basis through the transformations, $T_2$ and $T_3$,
\begin{align*}
    h_{I \times_t J}^{\vec{\epsilon}}(\mathbf{x}_1, \mathbf{x}_2) := h_{I \times J}^{\vec{\epsilon}}(T_2(\mathbf{x}_1, \mathbf{x}_2)) \quad\text{and}\quad h_{I \,_t\!\times J}^{\vec{\epsilon}}(\mathbf{x}_1, \mathbf{x}_2) := h_{I \times J}^{\vec{\epsilon}}(T_3(\mathbf{x}_1, \mathbf{x}_2)).
\end{align*}

\smallskip
\begin{prop}\label{prop:haar_properties}
Let 
$$\mathcal{B}_2 = \{ h_{I \times_t J}^{\vec{\epsilon}} \}_{I,J \in \mathscr{D}, \, \vec{\epsilon}}  \qquad\text{and}\qquad \mathcal{B}_3 = \{ h_{I \,_t\!\times J}^{\vec{\epsilon}} \}_{I,J \in \mathscr{D}, \, \vec{\epsilon}}$$ 
be the twisted Haar systems. Then both $\mathcal{B}_2$ and $\mathcal{B}_3$ satisfy the following properties:
\begin{itemize}
    \item[(i)] The support of each basis function  matches the corresponding twisted dyadic tube. Specifically,
    \begin{align*}
        {\rm supp}\big(h_{I \times_t J}^{\vec{\epsilon}}\big) = I \times_t J\quad\text{and}\quad
        {\rm supp}\big(h_{I \,_t\!\times J}^{\vec{\epsilon}}\big) = I \,_t\!\times J,
    \end{align*}
    with Lebesgue measures $|I \times_t J| = |I \,_t\!\times J| \simeq |I||J|$.

    \smallskip
    
    \item[(ii)] Both $\mathcal{B}_2$ and $\mathcal{B}_3$ are complete orthonormal bases   in $L^2(\mathbb{R}^{2m})$.

    \smallskip
    
    \item[(iii)]   The elements of $\mathcal B_2$ have cancellation in the two variables
    \[
    u=x_1-x_2,
    \qquad
    v=x_2,
    \]
    while the elements of $\mathcal B_3$ have cancellation in the two variables
    \[
    u=x_1,
    \qquad
    v=x_2-x_1.
    \]
    More precisely, for $R=I\times_t J\in\mathscr R^{(2)}$,
    \[
    h_R^{\vec\epsilon}(x_1,x_2)
    =
    h_I^{\epsilon_1}(x_1-x_2)\,h_J^{\epsilon_2}(x_2),
    \]
    so the integral vanishes separately in the $u$-variable and in the $v$-variable; similarly for $\mathcal B_3$.
\end{itemize}
\end{prop}

\begin{proof}
 The cancellation property and orthogonality follow from the cancellation property and the orthogonality of the standard Haar functions and the definition of the measure-preserving transformations.  Concretely, we only consider $\mathcal B_2$, and the proof for $\mathcal B_3$ is identical.

Since
\[
T_2(x_1,x_2)=(x_1-x_2,x_2)
\]
is a linear bijection with determinant $1$, it is measure-preserving. Therefore, for
$R=I\times_t J=T_2^{-1}(I\times J)$,
\[
\text{supp\,}(h_R^{\vec\epsilon})
=
T_2^{-1}(\text{supp\,}(h_{I\times J}^{\vec\epsilon}))
=
T_2^{-1}(I\times J)
=
R,
\]
and
\[
|R|
=
|T_2^{-1}(I\times J)|
=
|I\times J|
=
|I||J|.
\]

Let $U_2\phi:=\phi\circ T_2$. Since $T_2$ is measure-preserving, $U_2$ is unitary on
$L^2(\mathbb R^{2m})$. By definition,
\[
h_R^{\vec\epsilon}=U_2(h_{I\times J}^{\vec\epsilon}),
\]
so $\mathcal B_2$ is the unitary image of the standard tensor-product Haar basis $\mathcal B_1$.
Hence $\mathcal B_2$ is again a complete orthonormal basis of $L^2(\mathbb R^{2m})$.

Finally,
\[
h_R^{\vec\epsilon}(x_1,x_2)
=
h_I^{\epsilon_1}(x_1-x_2)\,h_J^{\epsilon_2}(x_2),
\]
which makes the cancellation in the rectified variables $(u,v)=(x_1-x_2,x_2)$ explicit.

 \smallskip

 It remains to verify that both $\mathcal{B}_2$ and $\mathcal{B}_3$ are complete in $L^2(\mathbb{R}^{2m})$.  Without loss of generality, we show that $\mathcal{B}_2$ is a complete orthonormal basis. To see this, it suffices to show that the only function orthogonal to every element in $\mathcal{B}_2$ is the zero function. Suppose $f \in L^2(\mathbb{R}^{2m})$ is a function such that
\begin{align*}
    \big\langle f, h_{I \times_t J}^{\vec{\epsilon}} \big\rangle = 0,\quad \forall I, J \in \mathscr{D} \,\,\text{and}\,\,\forall \vec{\epsilon}.
\end{align*}
Let $U_2$ be the unitary pullback operator associated with $T_2$, defined by $U_2 \phi = \phi \circ T_2$. By definition, we have $h_{I \times_t J}^{\vec{\epsilon}} = U_2(h_{I \times J}^{\vec{\epsilon}})$. Therefore,
\begin{align*}
    0 = \big\langle f, U_2 \big( h_{I \times J}^{\vec{\epsilon}} \big) \big\rangle = \big\langle U_2^* f, h_{I \times J}^{\vec{\epsilon}} \big\rangle.
\end{align*}
This implies that the function $U_2^* f$ is orthogonal to every element of the standard product Haar basis $\mathcal{B}_1$. As a result, $U_2^* f = 0$ almost everywhere. 

Since $U_2^*$ is a unitary operator, its kernel is trivial. Thus, we obtain that $f = 0$ almost everywhere. This proves that the span of $\mathcal{B}_2$ is dense in $L^2(\mathbb{R}^{2m})$, making it a complete orthonormal basis. The proof is complete.
\end{proof}


Let $\mathcal{B}_2 := \{ h_{I \times_t J}^{\vec{\epsilon}} \}$ and $\mathcal{B}_3 := \{ h_{I \,_t\!\times J}^{\vec{\epsilon}} \}$ be given in Proposition \ref{prop:haar_properties}. Then, as an immediate corollary, we obtain the following representation.

\begin{cor}[Twisted Haar frame representation]\label{thm:haar_frame}~\\
 The union of the three dyadic tube Haar systems $\mathcal{B}_1 \cup \mathcal{B}_2 \cup \mathcal{B}_3$ forms a tight frame for $L^2(\mathbb{R}^{2m})$ with frame bound $3$. Consequently, for any $f \in L^2(\mathbb{R}^{2m})$, we have \begin{align}\label{eq:haar_rep}
       f
=
\frac13
\sum_{i=1}^3
\ \sum_{R\in\mathscr R^{(i)}}
\ \sum_{\vec\epsilon\in\mathcal E\times\mathcal E}
\langle f,h_R^{\vec\epsilon}\rangle\,h_R^{\vec\epsilon},
    \end{align}
    where the sum converges unconditionally in $L^2(\mathbb{R}^{2m})$.
\end{cor}

\begin{proof}
For each $i=1,2,3$, the system $\mathcal B_i$ is a complete orthonormal basis of
$L^2(\mathbb R^{2m})$. Hence Parseval's identity holds for each $\mathcal B_i$:
\[
\sum_{R\in\mathscr R^{(i)}}\ \sum_{\vec\epsilon}
|\langle f,h_R^{\vec\epsilon}\rangle|^2
=
\|f\|_2^2.
\]
Summing over $i=1,2,3$ yields the tight frame identity with frame bound $3$, and
\eqref{eq:haar_rep} is the standard reconstruction formula for a tight frame.
\end{proof}

\section{Twisted multi-parameter filtrations and martingale properties}

The geometric affine shears $T_2$ and $T_3$ introduced in (\ref{affine}) allow us to  transfer the classical multi-parameter dyadic martingale theory to our twisted tube systems. We define a bi-parameter scale index $\mathbf{k} = (k_1, k_2) \in \mathbb{Z}^2$. Let $\mathscr{D}_{k}$ denote the collection of all standard dyadic cubes in $\mathbb{R}^m$ of side length $2^{-k}$. 

Define the standard bi-parameter dyadic filtration. Let $\mathscr{F}_{\mathbf{k}}^{(1)}$ be the $\sigma$-algebra generated by the standard dyadic rectangles at scale $\mathbf{k}$:
\begin{align*}
    \mathscr{F}_{\mathbf{k}}^{(1)} := \sigma\Big( \big\{ I \times J \subset \mathbb{R}^{2m} : I \in \mathscr{D}_{k_1}, J \in \mathscr{D}_{k_2} \big\} \Big).
\end{align*}
The standard conditional expectation operator $E_{\mathbf{k}}^{(1)} : L^1_{loc}(\mathbb{R}^{2m}) \to L^1_{loc}(\mathbb{R}^{2m})$ with respect to $\mathscr{F}_{\mathbf{k}}^{(1)}$ is the average of a function over these rectangles, that is
\begin{align*}
    E_{\mathbf{k}}^{(1)} f(\mathbf{x}) := \sum_{I \in \mathscr{D}_{k_1}} \sum_{J \in \mathscr{D}_{k_2}} \left( \frac{1}{|I \times J|} \int_{I \times J} f(\mathbf{y}) \, d\mathbf{y} \right) \chi_{I \times J}(\mathbf{x}).
\end{align*}

The {\it twisted filtrations} are defined by pulling back the standard $\sigma$-algebras through our geometric shears $T_2$ and $T_3$; in other words,
\begin{align*}
\mathscr{F}_{\mathbf{k}}^{(2)} := T_2^{-1}\big(\mathscr{F}_{\mathbf{k}}^{(1)}\big) = \sigma\Big( \big\{ I \times_t J : I \in \mathscr{D}_{k_1}, J \in \mathscr{D}_{k_2} \big\} \Big)
\end{align*}
and
\begin{align*}
\mathscr{F}_{\mathbf{k}}^{(3)} := T_3^{-1}\big(\mathscr{F}_{\mathbf{k}}^{(1)}\big) = \sigma\Big( \big\{ I \,_t\!\times J : I \in \mathscr{D}_{k_1}, J \in \mathscr{D}_{k_2} \big\} \Big).
\end{align*}
By definition, $\mathscr{F}_{\mathbf{k}}^{(2)}$ is generated exactly by the slanted tubes with base on the first direction  and $\mathscr{F}_{\mathbf{k}}^{(3)}$ is generated by the slanted tubes with base on the second direction.

 The twisted bi-parameter martingale difference operators are defined by conjugating the standard product martingale structure through the measure-preserving shears.  This makes the commuting structure transparent and aligns the twisted setting with the classical product filtration.

\begin{prop}[Twisted conditional expectations and martingale differences]\label{prop:twisted_martingale}~\\
    Let $U_2$ and $U_3$ be the pullback operators defined by $U_i f = f \circ T_i$ for each $i=2,3$. 
    \begin{itemize}
        \item[(i)] The conditional expectation operators for the twisted filtrations are  the unitary conjugations of the standard expectation,
        \begin{align}\label{Twcondexp}
            E_{\mathbf{k}}^{(i)} = U_i E_{\mathbf{k}}^{(1)} U_i^* \quad \  \     \text{for } i \in \{2, 3\}.
        \end{align}
        Explicitly, $E_{\mathbf{k}}^{(i)} f(\mathbf{x})$ computes the average of $f$ over the unique slanted tube in system $i$ containing $\mathbf{x}$.

        \smallskip
        
        \item[(ii)] The twisted bi-parameter martingale difference operator at scale $\mathbf{k}=(k_1,k_2)$ is defined in the usual product manner by taking successive differences in the two parameter directions:
$$
\Delta_{\mathbf{k}}^{(i)}
:=
(E_{k_1+1, k_2}^{(i)} - E_{k_1, k_2}^{(i)})
(E_{k_1, k_2+1}^{(i)} - E_{k_1, k_2}^{(i)}).
$$
Moreover, $\Delta_{\mathbf{k}}^{(i)}$ is the orthogonal projection onto the twisted Haar bases at scale $\mathbf{k}$ can be written as
        \begin{align}\label{twisteddifference}
            \Delta_{\mathbf{k}}^{(i)} f = \sum_{\substack{R \in \mathcal{B}_i \\ \text{scale}(R) = \mathbf{k}}} \sum_{\vec{\epsilon}} \big\langle f, h_R^{\vec{\epsilon}} \big\rangle h_R^{\vec{\epsilon}}.
        \end{align}
    \end{itemize}
\end{prop}

\begin{proof}
   We first prove (\ref{Twcondexp}) for $i=2$. The case for $i=3$ is parallel. Let $R_{std} = I \times J$ be a standard rectangle, then
    \begin{align*}
        \big(U_2 E_{\mathbf{k}}^{(1)} U_2^* f\big)(\mathbf{x}) &= E_{\mathbf{k}}^{(1)} \big(f \circ T_2^{-1}\big)(T_2 \mathbf{x}) \\
        &= \sum_{R_{std}} \left( \frac{1}{|R_{std}|} \int_{R_{std}} f(T_2^{-1} \mathbf{y}) \, d\mathbf{y} \right) \chi_{R_{std}}(T_2 \mathbf{x}).
    \end{align*}
    By perform the change of variables $\mathbf{u} = T_2^{-1} \mathbf{y}$ in the integral and the measure-preserving property, we obtain that 
    \begin{align*}
        \big(U_2 E_{\mathbf{k}}^{(1)} U_2^* f\big)(\mathbf{x}) &= \sum_{R_{slant}} \left( \frac{1}{|R_{slant}|} \int_{R_{slant}} f(\mathbf{u}) \, d\mathbf{u} \right) \chi_{R_{slant}}(\mathbf{x}).
    \end{align*}
    This is the definition of the conditional expectation $E_{\mathbf{k}}^{(2)} f(\mathbf{x})$ over the $\sigma$-algebra $\mathscr{F}_{\mathbf{k}}^{(2)}$.

Next, we show (\ref{twisteddifference}) for $i=2$. The case for $i=3$ is similar. Recall that in the standard bi-parameter setting ($i=1$), the multi-parameter martingale difference $\Delta_{\mathbf{k}}^{(1)}$ isolates the high-frequency oscillations precisely at scale $\mathbf{k}$. It is a classical result that $\Delta_{\mathbf{k}}^{(1)} f$ is the projection of $f$ onto the standard tensor-product Haar functions at scale $\mathbf{k}$, which is 
    \begin{align*}
        \Delta_{\mathbf{k}}^{(1)} f = \sum_{\substack{R \in \mathcal{B}_1 \\ \text{scale}(R) = \mathbf{k}}} \sum_{\vec{\epsilon}} \big\langle f, h_R^{\vec{\epsilon}} \big\rangle h_R^{\vec{\epsilon}}.
    \end{align*}
 Since \eqref{Twcondexp} holds for all relevant partial conditional expectations, the same conjugation identity passes to their differences. Thus
\[
\Delta_{\mathbf k}^{(2)}
=
U_2\,\Delta_{\mathbf k}^{(1)}\,U_2^*.
\] 
Therefore,
    \begin{align}\label{unitarykey}
        \Delta_{\mathbf{k}}^{(2)} f &= U_2 \Delta_{\mathbf{k}}^{(1)} U_2^* f \\
        &= U_2 \,\bigg( \sum_{\substack{R \in \mathcal{B}_1 \\ \text{scale}(R) = \mathbf{k}}} \sum_{\vec{\epsilon}} \big\langle U_2^* f, h_R^{\vec{\epsilon}} \big\rangle h_R^{\vec{\epsilon}} \bigg).\notag
    \end{align}
    By the definition of the adjoint, the inner product is $\langle U_2^* f, h_R^{\vec{\epsilon}} \rangle = \langle f, U_2 h_R^{\vec{\epsilon}} \rangle$. Applying the outer unitary operator $U_2$ to the Haar function gives exactly our slanted Haar function. As a result, we have
    \begin{align*}
        \Delta_{\mathbf{k}}^{(2)} f = \sum_{\substack{R_{slant} \in \mathcal{B}_2 \\ \text{scale}(R_{slant}) = \mathbf{k}}} \sum_{\vec{\epsilon}} \big\langle f, h_{R_{slant}}^{\vec{\epsilon}} \big\rangle h_{R_{slant}}^{\vec{\epsilon}}.
    \end{align*}
    This establishes the exact discrete Haar projection for the twisted martingales and the proof is complete.
\end{proof}

The unitary structure of these point transformations is powerful, as it extends beyond $L^2$ into all $L^p$ spaces. We can now define the discrete Littlewood--Paley square functions for our twisted geometry and trivially deduce their $L^p$ boundedness.

Let $f \in L^p(\mathbb{R}^{2m})$. We define the \textit{twisted martingale square functions} associated with the three tube systems as:
\begin{align*}
    S_{mart}^{(i)}(f)(\mathbf{x}) := \left( \sum_{\mathbf{k} \in \mathbb{Z}^2} \big| \Delta_{\mathbf{k}}^{(i)} f(\mathbf{x}) \big|^2 \right)^{1/2}, \quad \text{for } i \in \{1, 2, 3\}.
\end{align*}

\begin{thm}[Twisted Burkholder--Davis--Gundy inequalities]\label{thm:twisted_BDG}~\\
    For any $1 < p < \infty$, the twisted martingale square functions are bounded on $L^p(\mathbb{R}^{2m})$. Furthermore, we have the norm equivalences:
    \begin{align*}
        \big\| S_{mart}^{(i)}(f) \big\|_{L^p(\mathbb{R}^{2m})} \simeq \|f\|_{L^p(\mathbb{R}^{2m})} \quad \text{for } \  \    i \in \{1, 2, 3\},
    \end{align*}
    where the implicit constants depend only on $p$ and the dimension $m$.
\end{thm}

\begin{proof}
    For $i=1$, $S_{mart}^{(1)}(f)$ is the standard multi-parameter dyadic square function. The equivalence $\| S_{mart}^{(1)}(f) \|_p \simeq \|f\|_p$ is the classical bi-parameter Burkholder--Davis--Gundy inequality (or discrete Littlewood--Paley theorem) established by Fefferman and Stein.

    For $i=2$, we analyze the action of the unitary operator. Since $U_2 f(\mathbf{x}) = f(T_2 \mathbf{x})$ and the transformation $T_2$ is a measure-preserving bijection, thus for $1 \le p \le \infty$
    \begin{align*}
        \| U_2 f \|_p^p = \int_{\mathbb{R}^{2m}} |f(T_2 \mathbf{x})|^p \, d\mathbf{x} = \int_{\mathbb{R}^{2m}} |f(\mathbf{y})|^p \, d\mathbf{y} = \|f\|_p^p.
    \end{align*}
 From (\ref{unitarykey}), we get $\Delta_{\mathbf{k}}^{(2)} f = U_2 \Delta_{\mathbf{k}}^{(1)} U_2^* f$, and hence
    \begin{align*}
        S_{mart}^{(2)}(f)(\mathbf{x}) &= \left( \sum_{\mathbf{k} \in \mathbb{Z}^2} \big| \big(U_2 \Delta_{\mathbf{k}}^{(1)} U_2^* f\big)(\mathbf{x}) \big|^2 \right)^{1/2} \\
        &= \left( \sum_{\mathbf{k} \in \mathbb{Z}^2} \big| \big(\Delta_{\mathbf{k}}^{(1)} (U_2^* f)\big)(T_2 \mathbf{x}) \big|^2 \right)^{1/2} \\
        &= \big( U_2 S_{mart}^{(1)}(U_2^* f) \big)(\mathbf{x}).
    \end{align*}
    Taking the $L^p$ norm of both sides and applying the $L^p$ isometry of $U_2$ and $U_2^*$:
    \begin{align*}
        \big\| S_{mart}^{(2)}(f) \big\|_p = \big\| U_2 \big( S_{mart}^{(1)}(U_2^* f) \big) \big\|_p = \big\| S_{mart}^{(1)}(U_2^* f) \big\|_p \simeq \big\| U_2^* f \big\|_p = \|f\|_p.
    \end{align*}
    The exact same $L^p$ isometric transfer applies to system $i=3$ via the transformation $T_3$. Thus, the square function equivalences hold uniformly for all three twisted geometric systems.
\end{proof}

\section{Twisted dyadic square functions and Littlewood--Paley equivalences}

With the complete orthonormal twisted Haar bases established, we can define the discrete dyadic square functions adapted to the twisted geometry. In classical multi-parameter harmonic analysis, the standard dyadic square function isolates the orthogonal oscillations of a function at every scale and location using the standard Haar basis. This discrete framework avoids the complexities of continuous singular integrals and relies purely on martingale projections.   This is simply the Haar coefficient form of the martingale square functions introduced in Section 3.

For a function $f \in L^p(\mathbb{R}^{2m})$, we define the standard bi-parameter dyadic square function using the standard Haar basis $\mathcal{B}_1$. The operator measures the sum of the squared magnitudes of the orthogonal projections of $f$ onto every standard dyadic rectangle $R$, that is
\begin{align*}
    S_d^{(1)}(f)(\mathbf{x}) := \left( \sum_{R \in \mathcal{B}_1} \sum_{\vec{\epsilon}} \big| \langle f, h_R^{\vec{\epsilon}} \rangle h_R^{\vec{\epsilon}}(\mathbf{x}) \big|^2 \right)^{1/2},
\end{align*}
where the summation runs over all standard rectangles $R = I \times J \in \mathscr{D} \times \mathscr{D}$, and $\vec{\epsilon} \in \mathcal{E} \times \mathcal{E}$ indexes the internal Haar variations.

We define the \textit{twisted dyadic square functions} by projecting $f$ onto the slanted Haar bases $\mathcal{B}_2$ and $\mathcal{B}_3$, which we constructed via the affine shears $T_2$ and $T_3$. For $i \in \{2, 3\}$, the corresponding twisted dyadic square function is defined by
\begin{align*}
    S_d^{(i)}(f)(\mathbf{x}) := \left( \sum_{R_{slant} \in \mathcal{B}_i} \sum_{\vec{\epsilon}} \big| \langle f, h_{R_{slant}}^{\vec{\epsilon}} \rangle h_{R_{slant}}^{\vec{\epsilon}}(\mathbf{x}) \big|^2 \right)^{1/2}.
\end{align*}
The spatial support of the Haar functions $h_{R_{slant}}^{\vec{\epsilon}} \in \mathcal{B}_2$  matches the type II and  
II tubes, $S_d^{(2)}$ precisely measures the function's high-frequency oscillations along that specific twisted tube geometry. Similarly, $S_d^{(3)}$ measures oscillations strictly over the type IV 
and V 
tubes.

The forthcoming lemma gives us the structure of $S_d^{(2)}$ and $S_d^{(3)}$.

\begin{lem}\label{lem:dyadic_intertwine}
    For any $f \in L^p(\mathbb{R}^{2m})$ and $i \in \{2, 3\}$, the twisted dyadic square function of $f$ evaluated at $\mathbf{x}$ is identically the standard dyadic square function of the adjoint $U_i^* f$ evaluated at the shifted point $T_i \mathbf{x}$, that is
    \begin{align*}
        S_d^{(i)}(f)(\mathbf{x}) = S_d^{(1)}\big(U_i^* f\big)(T_i \mathbf{x}).
    \end{align*}
\end{lem}

\begin{proof}
    We prove the identity for $i=2$. Recall that the twisted Haar basis is the unitary image of the standard Haar basis, $\mathcal{B}_2 = \{ U_2 h_{R_{std}}^{\vec{\epsilon}} \}$, where $R_{std} \in \mathcal{B}_1$. Therefore, we have
    \begin{align*}
        \langle f, h_{R_{slant}}^{\vec{\epsilon}} \rangle h_{R_{slant}}^{\vec{\epsilon}}(\mathbf{x}) &= \big\langle f, U_2 h_{R_{std}}^{\vec{\epsilon}} \big\rangle \big( U_2 h_{R_{std}}^{\vec{\epsilon}} \big)(\mathbf{x})\\
        &= \big\langle U_2^* f, h_{R_{std}}^{\vec{\epsilon}} \big\rangle\big(h_{R_{std}}^{\vec{\epsilon}}(T_2 \mathbf{x})\big).
    \end{align*}
    Substituting these two identities back into the twisted square function expansion, then
    \begin{align*}
        S_d^{(2)}(f)(\mathbf{x}) &= \left( \sum_{R_{std} \in \mathcal{B}_1} \sum_{\vec{\epsilon}} \big| \big\langle U_2^* f, h_{R_{std}}^{\vec{\epsilon}} \big\rangle h_{R_{std}}^{\vec{\epsilon}}(T_2 \mathbf{x}) \big|^2 \right)^{1/2}=S_d^{(1)}\big(U_2^* f\big)(T_2 \mathbf{x}),
    \end{align*}
which gives the desired result. 
\end{proof}

\begin{thm}[Twisted dyadic Littlewood--Paley equivalences]\label{thm:twisted_dyadic_LP}~\\
    For any $1 < p < \infty$, the twisted dyadic square functions $S_d^{(2)}$ and $S_d^{(3)}$ are bounded operators on $L^p(\mathbb{R}^{2m})$. Furthermore, we have the norm equivalences:
    \begin{align*}
        \big\| S_d^{(i)}(f) \big\|_{L^p(\mathbb{R}^{2m})} \simeq \|f\|_{L^p(\mathbb{R}^{2m})}, \quad \text{for } i \in \{1, 2, 3\},
    \end{align*}
    where the implicit constants depend only on $p$ and the dimension $m$.
\end{thm}

\begin{proof}
    The equivalence $\| S_d^{(1)}(f) \|_p \simeq \|f\|_p$ is the classical bi-parameter dyadic Littlewood--Paley theorem established by Fefferman and Stein.
    
    For the twisted cases $i \in \{2, 3\}$, we utilize the exact pointwise intertwining identity from Lemma \ref{lem:dyadic_intertwine}. Since 
    \begin{align*}
        \big\| S_d^{(i)}(f) \big\|_{L^p(\mathbb{R}^{2m})}^p &= \int_{\mathbb{R}^{2m}} \big| S_d^{(i)}(f)(\mathbf{x}) \big|^p \, d\mathbf{x} 
        = \int_{\mathbb{R}^{2m}} \Big| S_d^{(1)}\big(U_i^* f\big)(T_i \mathbf{x}) \Big|^p \, d\mathbf{x},
    \end{align*}
  then from the change of variables $\mathbf{y} = T_i \mathbf{x}$, 
    \begin{align*}
        \int_{\mathbb{R}^{2m}} \Big| S_d^{(1)}\big(U_i^* f\big)(T_i \mathbf{x}) \Big|^p \, d\mathbf{x} &= \int_{\mathbb{R}^{2m}} \Big| S_d^{(1)}\big(U_i^* f\big)(\mathbf{y}) \Big|^p \, d\mathbf{y} 
        = \big\| S_d^{(1)}\big(U_i^* f\big) \big\|_{L^p(\mathbb{R}^{2m})}^p.
    \end{align*}
   Hence, from the classical Fefferman--Stein bi-parameter Littlewood--Paley equivalence and the measure-preserving property of $U_i^*$, 
    \begin{align*}
        \big\| S_d^{(1)}\big(U_i^* f\big) \big\|_p \simeq \big\| U_i^* f \big\|_p= \| f \|_p,\quad\forall\, 1<p<\infty.
    \end{align*}
     This completes the  proof by combining these equivalences together.
\end{proof}

\section{Projected Shards and Admissible Dyadic Structures\\ on the Quotient Group $\mathscr N$}
\label{sec:quotient_shards}
\setcounter{equation}{0}

 In the quotient setting we need to distinguish between two related objects. The first is the family of sets obtained directly from projecting products of Heisenberg tiles; these are the \emph{raw projected shards}, and they reflect the quotient geometry itself. The second is a dyadic family suitable for Haar expansions and discrete square functions. For this purpose we introduce the \emph{analytic dyadic shards}, which are exact inverse images of product blocks under measure-preserving rectifications. They are chosen so as to remain uniformly comparable to the raw projected shards in the relevant regimes. See also \cite{MR4808333}.
\medskip

We write points of the quotient group $\mathscr N$ as
\[
\mathbf g=(\mathbf z,t_1,t_2)
=(\mathbf z_1,\mathbf z_2,\mathbf z_3,t_1,t_2),
\qquad\text{where}\,\,
\mathbf z_\mu\in \mathbb C^{n_\mu},\,\, t_1,t_2\in \mathbb R,
\]
and points of the lifted group $\tilde{\mathscr N}=\mathscr H_1\times
\mathscr H_2\times \mathscr H_3$ as
\[
\tilde{\mathbf g}
=
\big((\mathbf z_1,u_1),(\mathbf z_2,u_2),(\mathbf z_3,u_3)\big).
\]
The quotient homomorphism is given by
\[
\pi(\tilde{\mathbf g})
=
(\mathbf z_1,\mathbf z_2,\mathbf z_3,u_1+u_3,u_2+u_3).
\]

\begin{remark} 
Throughout this section, the notation $T(\mathbf g,2^{\mathbf j})$ refers to the quotient tube of scale $\mathbf j=(j_1,j_2,j_3)$ centered at $\mathbf g$, in the sense of the three twisted quotient regimes arising in the continuous theory. Concretely, the relevant central geometries are:
\begin{itemize}
\item[(i)] product-type with central scales $(2^{2j_1},2^{2j_2})$ when $j_3<j_2$;

\smallskip

\item[(ii)] product-type with central scales $(2^{2j_1},2^{2j_3})$ when $j_2<j_3<j_1$;

\smallskip

\item[(iii)] slanted-type with rectified central scales $(2^{2j_1},2^{2j_3})$ in the coordinates $(u,v)=(t_1-t_2,t_2)$ when $j_1<j_3$.
\end{itemize}
More precisely, \(T(\mathbf g,2^{\mathbf j})\) may denote any model region in the corresponding regime with the stated spatial and central scales; throughout this section it is used only up to uniform comparability.
All such comparisons are understood up to multiplicative constants depending only on the dimensions and on the fixed stacking parameter $\kappa$.
\end{remark}

Throughout the dyadic constructions below, all cubes, intervals, and product blocks are understood in the standard half-open sense. Accordingly, all partitions are exact up to boundary sets of measure zero.

\subsection{Heisenberg tiles and stacked tiles}

For each $\mu\in\{1,2,3\}$ let $\mathscr H_\mu=\mathbb C^{n_\mu}\times \mathbb R$
be the Heisenberg factor with anisotropic dilation
\[
\delta_a(\mathbf z_\mu,u_\mu)=(a\mathbf z_\mu,a^2u_\mu),
\qquad a>0.
\]
Recall from \cite{St,Ty} that there exists a basic tile
\[
T_{o}^{(\mu)}
=
\Big\{(\mathbf y_\mu,s_\mu)\in \mathscr H_\mu:
\mathbf y_\mu\in [0,1)^{2n_\mu},\,
f_{o,\mu}(\mathbf y_\mu)\le s_\mu<f_{o,\mu}(\mathbf y_\mu)+1
\Big\},
\]
where $f_{o,\mu}$ is bounded and continuous on $[0,1)^{2n_\mu}$.  Let
$\mathfrak T_{j_\mu}^{(\mu)}$ be the collection of all translates and dilates
of $T_o^{(\mu)}$ at scale $j_\mu$.  For each fixed $j_\mu$, the family
$\mathfrak T_{j_\mu}^{(\mu)}$ is a partition of $\mathscr H_\mu$, and each tile
is comparable to a Heisenberg ball of radius $2^{j_\mu}$.

Choose a large integer $\kappa\ge 10$ such that
\[
\|f_{o,\mu}\|_{L^\infty([0,1)^{2n_\mu})}\le 2^{\kappa-10},
\qquad \mu=1,2,3.
\]
For $\mu=1,2$, define the stacked basic tile
\[
S_o^{(\mu)}
:=
\bigcup_{m=0}^{2^\kappa-1}
\big(\mathbf 0,m\big)\,T_o^{(\mu)}.
\]
Then $S_o^{(\mu)}$ has the form
\[
S_o^{(\mu)}
=
\Big\{(\mathbf y_\mu,s_\mu):
\mathbf y_\mu\in [0,1)^{2n_\mu},\,
s_\mu\in f_{o,\mu}(\mathbf y_\mu)+[0,2^\kappa)
\Big\}.
\]
Its scale-$j_\mu$ versions are
\[
\delta_{2^{j_\mu}}S_o^{(\mu)}
=
\Big\{(\mathbf z_\mu,t_\mu):
\mathbf z_\mu\in [0,2^{j_\mu})^{2n_\mu},\,
t_\mu\in 2^{2j_\mu}f_{o,\mu}(2^{-j_\mu}\mathbf z_\mu)
+
2^{2j_\mu+\kappa}[0,1)
\Big\}.
\]

For later use, set
\[
\Box_{j_\mu}^{(\mu)+}:=[0,2^{j_\mu})^{2n_\mu},
\qquad
\Box_{\mathbf j}^+
:=
\Box_{j_1}^{(1)+}\times \Box_{j_2}^{(2)+}\times \Box_{j_3}^{(3)+},
\qquad
\mathbf j=(j_1,j_2,j_3)\in\mathbb Z^3.
\]

\begin{figure}[htbp]
\centering
\begin{tikzpicture}[
    x=1.15cm,y=1.15cm,
    axis/.style={-Latex,thick},
    tileline/.style={blue!70!black,thick},
    stackline/.style={orange!70!black,thick},
    tilefill/.style={fill=blue!18,draw=blue!70!black,thick},
    stackfill/.style={fill=orange!18,draw=orange!70!black,thick}
]

\begin{scope}[shift={(0,0)}]
    \draw[axis] (0,0) -- (3.4,0) node[right] {$z_\mu$};
    \draw[axis] (0,0) -- (0,3.2) node[above] {$u_\mu$};

    \fill[tilefill]
        (0.35,0.65)
        plot[smooth] coordinates {(0.35,0.65) (0.85,0.95) (1.40,0.82) (2.00,1.08) (2.75,0.92)}
        -- (2.75,1.72)
        plot[smooth] coordinates {(2.75,1.72) (2.00,1.88) (1.40,1.62) (0.85,1.75) (0.35,1.45)}
        -- cycle;

    \draw[tileline] plot[smooth] coordinates {(0.35,0.65) (0.85,0.95) (1.40,0.82) (2.00,1.08) (2.75,0.92)};
    \draw[tileline] plot[smooth] coordinates {(0.35,1.45) (0.85,1.75) (1.40,1.62) (2.00,1.88) (2.75,1.72)};

    \node at (1.55,1.28) {$T_o^{(\mu)}$};
    \node[align=center] at (1.65,-0.65)
    {\textbf{(a)} Basic Heisenberg tile\\ between $f_{o,\mu}$ and $f_{o,\mu}+1$};
\end{scope}

\begin{scope}[shift={(5.2,0)}]
    \draw[axis] (0,0) -- (3.4,0) node[right] {$z_\mu$};
    \draw[axis] (0,0) -- (0,4.0) node[above] {$u_\mu$};

    \fill[stackfill]
        (0.35,0.65)
        plot[smooth] coordinates {(0.35,0.65) (0.85,0.95) (1.40,0.82) (2.00,1.08) (2.75,0.92)}
        -- (2.75,3.05)
        plot[smooth] coordinates {(2.75,3.05) (2.00,3.22) (1.40,2.95) (0.85,3.08) (0.35,2.78)}
        -- cycle;

    \draw[stackline] plot[smooth] coordinates {(0.35,0.65) (0.85,0.95) (1.40,0.82) (2.00,1.08) (2.75,0.92)};
    \draw[stackline] plot[smooth] coordinates {(0.35,2.78) (0.85,3.08) (1.40,2.95) (2.00,3.22) (2.75,3.05)};

    \foreach \s in {0.45,0.90,1.35,1.80} {
        \draw[orange!55!black,thin]
        plot[smooth] coordinates
        {(0.35,0.65+\s) (0.85,0.95+\s) (1.40,0.82+\s) (2.00,1.08+\s) (2.75,0.92+\s)};
    }

{\color{blue}   \draw[decorate,decoration={brace,amplitude=5pt, mirror}] 
        (3.05,0.70) -- (3.05,2.95)
        node[midway,right=7pt] {$2^\kappa$};}

    \node at (1.55,1.95) {$S_o^{(\mu)}$};
    \node[align=center] at (1.65,-0.65)
    {\textbf{(b)} Stacked tile\\ same profile, enlarged central height};
\end{scope}

\end{tikzpicture}
\caption{Schematic picture of the basic Heisenberg tile $T_o^{(\mu)}$ and the stacked tile $S_o^{(\mu)}$. The horizontal axis represents one schematic spatial direction in $\mathbb C^{n_\mu}$, and the vertical axis represents the central variable. The stacked tile keeps the same irregular lower profile and enlarges the center direction by a factor of order $2^\kappa$.}
\label{fig:heisenberg_tile_stack}
\end{figure}
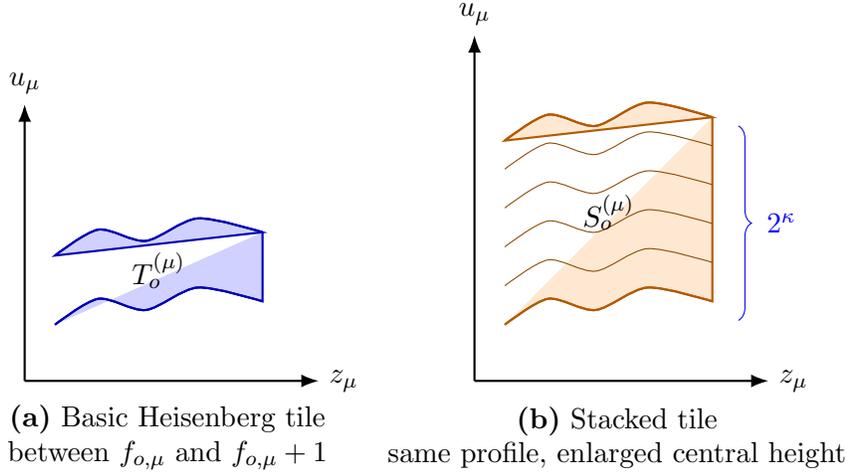

\subsection{Raw projected shards}

We now recall the quotient shard construction coming from products of stacked
tiles in the lifting group.  Throughout this subsection we impose the
convention
\[
j_2\le j_1.
\]
The opposite order is obtained by exchanging the first two Heisenberg factors
and replacing the shear in the last regime by its symmetric counterpart.

For a scale $\mathbf j=(j_1,j_2,j_3)\in \mathbb Z^3$, define the raw pre-shard
\begin{align}
\mathcal S_o^{\mathbf j}
&:=
\delta_{2^{j_1}}S_o^{(1)}\times \delta_{2^{j_2}}S_o^{(2)}
\times \Box_{j_3}^{(3)+}
\notag\\
&=
\Big\{(\mathbf z,t_1,t_2):
\mathbf z\in \Box_{\mathbf j}^+,\,
(t_1,t_2)\in \widehat I_{\mathbf z,\mathbf j}
\Big\},
\label{eq:raw_pre_shard}
\end{align}
where
\begin{align}
\widehat I_{\mathbf z,\mathbf j}
:=
\Big(
2^{2j_1}f_{o,1}(2^{-j_1}\mathbf z_1),
2^{2j_2}f_{o,2}(2^{-j_2}\mathbf z_2)
\Big)
+
2^{2j_1+\kappa}[0,1)\times 2^{2j_2+\kappa}[0,1).
\label{eq:raw_pre_shard_interval}
\end{align}

\begin{lem}[Elementary dyadic interval partition]\label{lem:interval-partition}~\\
Let $\ell>0$, and let $N\in\mathbb N$. Then the collection of intervals
\[
\Big\{[m\ell,(m+1)\ell):m\in\mathbb Z\Big\}
\]
forms a partition of $\mathbb R$.  More generally, for any fixed $a\in \mathbb Z$, the collection
\[
\Big\{[m\ell,(m+N)\ell):m\in a+N\mathbb Z\Big\}
\]
forms a partition of $\mathbb R$ into half-open intervals of length $N\ell$.  Consequently, finite Cartesian products of such interval families form partitions of the corresponding Euclidean space.
\end{lem}

\begin{proof}
The first statement is standard. For the second, the left endpoints form an arithmetic progression of step $N\ell$, hence the corresponding half-open intervals tile $\mathbb R$ without overlap except at endpoints. The product statement is immediate.
\end{proof}

\begin{lem}[Spatial partition at scale $\mathbf j$]\label{lem:spatial-partition}~\\
For each $\mathbf j=(j_1,j_2,j_3)\in \mathbb Z^3$, the translates of
\[
\Box_{\mathbf j}^{+}
=
[0,2^{j_1})^{2n_1}\times [0,2^{j_2})^{2n_2}\times [0,2^{j_3})^{2n_3}
\]
by the lattice
\[
2^{j_1}\mathbb Z^{2n_1}\times 2^{j_2}\mathbb Z^{2n_2}\times 2^{j_3}\mathbb Z^{2n_3}
\]
form a partition of $\mathbb C^{n_1+n_2+n_3}\cong \mathbb R^{2n_1+2n_2+2n_3}$.
\end{lem}

\begin{proof}
This is the Cartesian product of the standard half-open dyadic cube partitions in each spatial factor.
\end{proof}

\begin{lem}\label{lem:raw_pre_shard_sizes}
For each fixed $\mathbf z\in \Box_{\mathbf j}^+$, the central fiber of the raw pre-shard $\mathcal S_o^{\mathbf j}$ is a rectangle whose side lengths are exactly
\[
2^{2j_1+\kappa}
\qquad\text{and}\qquad
2^{2j_2+\kappa}.
\]
\end{lem}

\begin{proof}
This is immediate from \eqref{eq:raw_pre_shard_interval}.
\end{proof}

The actual raw projected basic shards are then obtained by taking unions of
translates of $\mathcal S_o^{\mathbf j}$ in the center variables.

\medskip
\noindent
\textbf{Case I: $j_3<j_2$.}
Define
\begin{align}
\mathcal P_o^{\mathbf j,(1)}
:=
\bigcup_{m_1,m_2\in\{0,-1\}}
\big(\mathbf 0,m_12^{2j_1+\kappa},m_22^{2j_2+\kappa}\big)\,
\mathcal S_o^{\mathbf j}.
\label{eq:raw_shard_case1}
\end{align}

\medskip
\noindent
\textbf{Case II: $j_2<j_3<j_1$.}
Let
\[
\Gamma_1
:=
\left\{
-2^{2(j_3-j_2)},\ldots,-1,0,1,\ldots,2^{2(j_3-j_2)}-1
\right\},
\]
and define
\begin{align}
\mathcal P_o^{\mathbf j,(2)}
:=
\bigcup_{\substack{m_1\in\{0,-1\}\\ m_2\in\Gamma_1}}
\big(\mathbf 0,m_12^{2j_1+\kappa},m_22^{2j_2+\kappa}\big)\,
\mathcal S_o^{\mathbf j}.
\label{eq:raw_shard_case2}
\end{align}

\medskip
\noindent
\textbf{Case III: $j_1<j_3$.}
First set
\[
\Gamma_2
:=
\left\{
-2^{2(j_1-j_2)},\ldots,-1,0,1,\ldots,2^{2(j_1-j_2)}-1
\right\},
\]
and define the intermediate block
\begin{align}
\widehat{\mathcal P}_o^{\mathbf j}
:=
\bigcup_{\substack{m_1\in\{0,-1\}\\ m_2\in\Gamma_2}}
\big(\mathbf 0,m_12^{2j_1+\kappa},m_22^{2j_2+\kappa}\big)\,
\mathcal S_o^{\mathbf j}.
\label{eq:raw_shard_case3hat}
\end{align}
Next add the two neighboring copies in the $t_1$-direction:
\begin{align}
\check{\mathcal P}_o^{\mathbf j}
:=
\big(\mathbf 0,-2\cdot 2^{2j_1+\kappa},0\big)\,\widehat{\mathcal P}_o^{\mathbf j}
\cup
\widehat{\mathcal P}_o^{\mathbf j}
\cup
\big(\mathbf 0,2\cdot 2^{2j_1+\kappa},0\big)\,\widehat{\mathcal P}_o^{\mathbf j}.
\label{eq:raw_shard_case3check}
\end{align}
Finally, let
\[
\Gamma_3
:=
\left\{
-2^{2(j_3-j_1)},\ldots,-2,0,2,\ldots,2^{2(j_3-j_1)}-2
\right\},
\]
and define
\begin{align}
\mathcal P_o^{\mathbf j,(3)}
:=
\bigcup_{m\in \Gamma_3}
\big(\mathbf 0,m2^{2j_1+\kappa},m2^{2j_1+\kappa}\big)\,
\check{\mathcal P}_o^{\mathbf j}.
\label{eq:raw_shard_case3}
\end{align}

 For later use, we make the translation lattices explicit.  For the spatial variables, define
\[
\Lambda_{\mathrm{sp}}(\mathbf j)
:=
2^{j_1}\mathbb Z^{2n_1}\times 2^{j_2}\mathbb Z^{2n_2}\times 2^{j_3}\mathbb Z^{2n_3}.
\]

In Case I, set
\[
L_1:=2^{2j_1+\kappa},
\qquad
L_2:=2^{2j_2+\kappa},
\]
and define the central translation lattice
\[
\Lambda_{\mathrm{cen}}^{(1)}(\mathbf j)
:=
(2L_1)\mathbb Z\times (2L_2)\mathbb Z.
\]

In Case II, set
\[
L_1:=2^{2j_1+\kappa},
\qquad
L_2:=2^{2j_2+\kappa},
\qquad
N:=2^{2(j_3-j_2)},
\]
and define
\[
\Lambda_{\mathrm{cen}}^{(2)}(\mathbf j)
:=
(2L_1)\mathbb Z\times (2NL_2)\mathbb Z.
\]

In Case III, set
\[
L:=2^{2j_1+\kappa},
\qquad
N:=2^{2(j_3-j_1)}.
\]
In the rectified coordinates
\[
(u,v):=(t_1-t_2,t_2),
\]
define
\[
\widetilde\Lambda_{\mathrm{cen}}^{(3)}(\mathbf j)
:=
(6L)\mathbb Z\times (2NL)\mathbb Z.
\]
Equivalently, in the original $(t_1,t_2)$-coordinates we define
\[
\Lambda_{\mathrm{cen}}^{(3)}(\mathbf j)
:=
\Theta^{-1}\big(\widetilde\Lambda_{\mathrm{cen}}^{(3)}(\mathbf j)\big),
\]
where $\Theta(t_1,t_2)=(t_1-t_2,t_2)$.

For $k=1,2,3$, we let $\mathfrak P_{\mathbf j}^{(k)}$ denote the family of all translates of the corresponding basic raw shard by
\[
\Lambda_{\mathrm{sp}}(\mathbf j)\times \Lambda_{\mathrm{cen}}^{(k)}(\mathbf j).
\]
In Case III, this means equivalently that after passing to the rectified coordinates
$(u,v)=(t_1-t_2,t_2)$, the central translates are taken with respect to
$\widetilde\Lambda_{\mathrm{cen}}^{(3)}(\mathbf j)$.

\begin{lem}[Central partition in Case I]\label{lem:center-partition-caseI}~\\
Assume that $j_3<j_2\le j_1$, and set
\[
L_1:=2^{2j_1+\kappa},
\qquad
L_2:=2^{2j_2+\kappa}.
\]
Then, for each fixed spatial translate of $\Box_{\mathbf j}^{+}$, the central fibers of the lattice translates of $\mathcal P_o^{\mathbf j,(1)}$ form a partition of $\mathbb R^2$. Each such central fiber is a rectangle of exact side lengths
\[
2L_1
\qquad\text{and}\qquad
2L_2.
\]
In particular, it is uniformly comparable to a rectangle of side lengths $L_1$ and $L_2$.
\end{lem}

\begin{proof}
Fix the spatial variables $\mathbf z\in \Box_{\mathbf j}^+$. By Lemma~\ref{lem:raw_pre_shard_sizes},
the central fiber of $\mathcal S_o^{\mathbf j}$ at $\mathbf z$ is a half-open rectangle
\[
I_{\mathbf z}^{(1)}\times I_{\mathbf z}^{(2)}
\]
with
\[
|I_{\mathbf z}^{(1)}|=L_1,
\qquad
|I_{\mathbf z}^{(2)}|=L_2.
\]
The union over $m_1,m_2\in\{0,-1\}$ in \eqref{eq:raw_shard_case1} therefore produces
\[
\big(I_{\mathbf z}^{(1)}-L_1\big)\cup I_{\mathbf z}^{(1)}
\quad\times\quad
\big(I_{\mathbf z}^{(2)}-L_2\big)\cup I_{\mathbf z}^{(2)},
\]
which is a half-open rectangle of exact side lengths $2L_1$ and $2L_2$.

Now translate these rectangles by the lattice
\[
\Lambda_{\mathrm{cen}}^{(1)}(\mathbf j)=(2L_1)\mathbb Z\times (2L_2)\mathbb Z.
\]
By Lemma~\ref{lem:interval-partition}, the translated fibers partition $\mathbb R^2$ up to boundary overlap only. Since this holds for every fixed spatial fiber, the conclusion follows.
\end{proof}

\begin{lem}[Central partition in Case II]\label{lem:center-partition-caseII}~\\
Assume that $j_2<j_3<j_1$, and set
\[
L_1:=2^{2j_1+\kappa},
\qquad
L_2:=2^{2j_2+\kappa},
\qquad
N:=2^{2(j_3-j_2)}.
\]
Then, for each fixed spatial translate of $\Box_{\mathbf j}^{+}$, the central fibers of the lattice translates of $\mathcal P_o^{\mathbf j,(2)}$ form a partition of $\mathbb R^2$ with exact period cell step size
\[
(2L_1,\ 2NL_2).
\]
Equivalently, each such central fiber is a rectangle of exact side lengths
\[
2L_1
\qquad\text{and}\qquad
2NL_2.
\]
In particular, it is uniformly comparable to a rectangle of side lengths $L_1$ and $NL_2$.
\end{lem}

\begin{proof}
Fix the spatial variables $\mathbf z\in \Box_{\mathbf j}^+$. As above, the first central direction is treated exactly as in Case I, so the union over $m_1\in\{0,-1\}$ produces a half-open interval of exact length $2L_1$.

For the second central direction, the basic interval has length $L_2$, and
\[
\Gamma_1=\{-N,-N+1,\dots,-1,0,1,\dots,N-1\}
\]
contains exactly $2N$ consecutive integers. Hence the intervals
\[
I_{\mathbf z}^{(2)}+m_2L_2,
\qquad m_2\in\Gamma_1,
\]
are $2N$ consecutive half-open intervals of common length $L_2$, and therefore their union is a single half-open interval of exact length
\[
(2N)L_2.
\]

Thus the full central fiber of $\mathcal P_o^{\mathbf j,(2)}$ is a half-open rectangle of exact side lengths
\[
2L_1
\qquad\text{and}\qquad
2NL_2.
\]
Translating by
\[
\Lambda_{\mathrm{cen}}^{(2)}(\mathbf j)=(2L_1)\mathbb Z\times (2NL_2)\mathbb Z
\]
partitions $\mathbb R^2$ by Lemma~\ref{lem:interval-partition}. This proves the claim.
\end{proof}

\begin{lem}[Central partition in Case III]\label{lem:center-partition-caseIII}~\\
Assume that $j_1<j_3$, and set
\[
L:=2^{2j_1+\kappa},
\qquad
N:=2^{2(j_3-j_1)}.
\]
Then, in the rectified coordinates
\[
(u,v):=(t_1-t_2,t_2),
\]
the central fibers of the lattice translates of $\mathcal P_o^{\mathbf j,(3)}$ form a partition of the $(u,v)$-plane by translates of a fixed half-open strip-type cell $\Omega_{L,N}$, with period lattice
\[
\widetilde\Lambda_{\mathrm{cen}}^{(3)}(\mathbf j)
=
(6L)\mathbb Z\times (2NL)\mathbb Z.
\]
Moreover, $\Omega_{L,N}$ is uniformly comparable to a rectangle of side lengths
\[
L
\qquad\text{and}\qquad
NL.
\]
\end{lem}

\begin{proof}
We divide the construction into three steps.

\smallskip
\noindent
\textit{Step 1: the intermediate block $\widehat{\mathcal P}_o^{\mathbf j}$.}
By Lemma~\ref{lem:raw_pre_shard_sizes}, the raw pre-shard $\mathcal S_o^{\mathbf j}$ has central side lengths
\[
2^{2j_1+\kappa}=L
\qquad\text{and}\qquad
2^{2j_2+\kappa}.
\]
Let
\[
M:=2^{2(j_1-j_2)}.
\]
Then
\[
\Gamma_2=\{-M,-M+1,\dots,-1,0,1,\dots,M-1\}
\]
contains exactly $2M$ consecutive integers. Hence in the second central direction the union in
\eqref{eq:raw_shard_case3hat} concatenates exactly $2M$ consecutive half-open intervals, each of length $2^{2j_2+\kappa}$. Therefore the resulting second central side length is
\[
(2M)\,2^{2j_2+\kappa}
=
2^{2(j_1-j_2)+1}\,2^{2j_2+\kappa}
=
2^{2j_1+\kappa+1}
=
2L.
\]

In the first central direction, the union over $m_1\in\{0,-1\}$ concatenates two adjacent half-open intervals, each of length $L$, so the resulting first central side length is also exactly
\[
2L.
\]
Consequently, for each fixed spatial fiber, $\widehat{\mathcal P}_o^{\mathbf j}$ has a central fiber which is a half-open rectangle of exact side lengths
\[
2L
\qquad\text{and}\qquad
2L.
\]

\smallskip
\noindent
\textit{Step 2: the enlarged block $\check{\mathcal P}_o^{\mathbf j}$.}
By definition,
\[
\check{\mathcal P}_o^{\mathbf j}
=
(-2L,0)+\widehat{\mathcal P}_o^{\mathbf j}
\ \cup\
\widehat{\mathcal P}_o^{\mathbf j}
\ \cup\
(2L,0)+\widehat{\mathcal P}_o^{\mathbf j}.
\]
Thus, for each fixed value of $t_2$, the corresponding $t_1$-fiber is the union of three adjacent half-open intervals, each of length $2L$, and hence has exact total length
\[
6L.
\]

Now pass to the rectified coordinates
\[
(u,v)=(t_1-t_2,t_2).
\]
For fixed $v=t_2$, one has $u=t_1-v$, so translation in the $t_1$-direction by $\pm 2L$ becomes translation in the $u$-direction by the same amount. Therefore, for each fixed $v$, the $u$-fiber of $\check{\mathcal P}_o^{\mathbf j}$ is a half-open interval of exact length $6L$.

\smallskip
\noindent
\textit{Step 3: the diagonal stacking.}
In the final union \eqref{eq:raw_shard_case3}, the index set
\[
\Gamma_3=\{-N,-N+2,\dots,-2,0,2,\dots,N-2\}
\]
contains exactly $N$ consecutive even integers. Under the rectifying map
\[
(u,v)=(t_1-t_2,t_2),
\]
the diagonal translation
\[
(t_1,t_2)\mapsto (t_1+mL,t_2+mL)
\]
becomes
\[
u\mapsto u,
\qquad
v\mapsto v+mL.
\]
Since consecutive values of $m$ differ by $2$, the corresponding $v$-step is exactly $2L$. Hence the union over $m\in\Gamma_3$ stacks exactly $N$ adjacent copies in the $v$-direction, producing a half-open interval of exact length
\[
N\cdot (2L)=2NL.
\]

Combining Steps 2 and 3, we obtain that in the $(u,v)$-coordinates each central fiber of
$\mathcal P_o^{\mathbf j,(3)}$ is a half-open strip-type cell of exact period sizes
\[
6L
\qquad\text{and}\qquad
2NL.
\]
Translating by the lattice
\[
\widetilde\Lambda_{\mathrm{cen}}^{(3)}(\mathbf j)
=
(6L)\mathbb Z\times (2NL)\mathbb Z
\]
therefore partitions the $(u,v)$-plane by Lemma~\ref{lem:interval-partition}. Denote one such period cell by $\Omega_{L,N}$.

Since the constants $6$ and $2$ are fixed, $\Omega_{L,N}$ is uniformly comparable to a rectangle of side lengths $L$ and $NL$. Finally, the change of variables $(t_1,t_2)\leftrightarrow (u,v)$ is linear and invertible, so it preserves the partition property up to boundary sets of measure zero.
\end{proof}

\begin{lem}\label{lem:raw_three_regimes}
The three raw projected basic shards match the three quotient tube regimes as follows.

\begin{itemize}
\item[(i)] If $j_3<j_2$, then $\mathcal P_o^{\mathbf j,(1)}$ has product-type central geometry uniformly comparable to central scales
\[
2^{2j_1+\kappa}
\qquad\text{and}\qquad
2^{2j_2+\kappa}.
\]

\item[(ii)] If $j_2<j_3<j_1$, then $\mathcal P_o^{\mathbf j,(2)}$ has product-type central geometry uniformly comparable to central scales
\[
2^{2j_1+\kappa}
\qquad\text{and}\qquad
2^{2j_3+\kappa}.
\]

\item[(iii)] If $j_1<j_3$, then $\mathcal P_o^{\mathbf j,(3)}$ has slanted central geometry; in the rectified coordinates $(u,v)=(t_1-t_2,t_2)$, it is uniformly comparable to central scales
\[
2^{2j_1+\kappa}
\qquad\text{and}\qquad
2^{2j_3+\kappa}.
\]
\end{itemize}
\end{lem}

\begin{proof}
This follows immediately from Lemmas~\ref{lem:center-partition-caseI}, \ref{lem:center-partition-caseII}, and \ref{lem:center-partition-caseIII}.
\end{proof}

\begin{prop}[Geometry of the raw projected shards]
\label{prop:raw_shards}~\\
For each fixed $\mathbf j\in\mathbb Z^3$ and each admissible case
$k\in\{1,2,3\}$ above, the family $\mathfrak P_{\mathbf j}^{(k)}$ is a
partition of $\mathscr N$.  Moreover, there exists an integer
$\sigma\ge 1$, depending only on the dimensions, such that every raw projected
shard $P\in \mathfrak P_{\mathbf j}^{(k)}$ satisfies
\begin{align}
T(\mathbf g\zeta_{\mathbf j},2^{\mathbf j-\sigma})
\subset P \subset
T(\mathbf g,2^{\mathbf j+\sigma}),
\label{eq:raw_shard_tube_comparison}
\end{align}
where $\mathbf g$ is the lattice center of $P$ and
\[
\zeta_{\mathbf j}
:=
(\zeta_{\mathbf j,1},\zeta_{\mathbf j,2},\zeta_{\mathbf j,3},0,0),
\qquad
\zeta_{\mathbf j,\mu}
:=
(\underbrace{2^{j_\mu-1},\ldots,2^{j_\mu-1}}_{2n_\mu\ \text{entries}}).
\]
\end{prop}

\begin{proof}
Fix an admissible regime $k\in\{1,2,3\}$ and a scale $\mathbf j$.

We first prove the partition statement. By Lemma~\ref{lem:spatial-partition}, the spatial translates of
\[
\Box_{\mathbf j}^+
=
[0,2^{j_1})^{2n_1}\times [0,2^{j_2})^{2n_2}\times [0,2^{j_3})^{2n_3}
\]
by the lattice $\Lambda_{\mathrm{sp}}(\mathbf j)$ partition the spatial variables.

Now fix one spatial translate of $\Box_{\mathbf j}^+$ and fix a spatial point $\mathbf z$ in that translate. For this fixed $\mathbf z$, the central fiber of the basic raw pre-shard is a half-open rectangle, depending on $\mathbf z$ through the profile functions
$f_{o,1}(2^{-j_1}\mathbf z_1)$ and $f_{o,2}(2^{-j_2}\mathbf z_2)$. In each of the three admissible regimes, Lemmas~\ref{lem:center-partition-caseI}, \ref{lem:center-partition-caseII}, and \ref{lem:center-partition-caseIII} show that the corresponding central translates of this fiber partition the central variables (in Case III, after rectification by $(u,v)=(t_1-t_2,t_2)$, and hence also in the original coordinates). Therefore, for each fixed spatial point $\mathbf z$, the central fibers of the translates in $\mathfrak P_{\mathbf j}^{(k)}$ partition $\mathbb R^2$ up to boundary overlap only.  Combining the spatial partition with this fiberwise central partition, we conclude that the family $\mathfrak P_{\mathbf j}^{(k)}$ partitions $\mathscr N$ up to boundary sets of measure zero.

We now turn to the tube comparison \eqref{eq:raw_shard_tube_comparison}. In all three regimes, the spatial component of a raw projected shard has side lengths exactly
\[
2^{j_1},\, 2^{j_2},\,\text{and}\,\,2^{j_3}
\]
in the three spatial factors. The corresponding central sizes are given by Lemmas~\ref{lem:center-partition-caseI}, \ref{lem:center-partition-caseII}, and \ref{lem:center-partition-caseIII}: in Case I they are of order
\[
2^{2j_1+\kappa}
\qquad\text{and}\qquad
2^{2j_2+\kappa},
\]
in Case II of order
\[
2^{2j_1+\kappa}
\qquad\text{and}\qquad
2^{2j_3+\kappa},
\]
and in Case III, after rectification, of order
\[
2^{2j_1+\kappa}
\qquad\text{and}\qquad
2^{2j_3+\kappa}.
\]
These are exactly the three model quotient tube geometries at scale $\mathbf j$, up to constants depending only on the dimensions and the fixed stacking parameter $\kappa$. Hence each raw projected shard contains a quotient tube of slightly smaller scale centered at its lattice midpoint, and is contained in a quotient tube of slightly larger scale centered at its lattice base point. This yields \eqref{eq:raw_shard_tube_comparison} for some uniform integer $\sigma\ge1$.
\end{proof}

\begin{remark}
The raw projected shards are the correct geometric descendants of the Heisenberg
tilings, but they are not the right objects for an exact tensor-product Haar
theory.  First, they are only a partition at each fixed scale; they are not
nested across scales.  Second, in Case III the diagonal union
\eqref{eq:raw_shard_case3} is comparable to a slanted parallelogram, but its
rectified shear image is not a genuine Cartesian block.  For the discrete Haar
construction we therefore replace the raw projected shards by exact inverse
images of product blocks under measure-preserving rectifications.
\end{remark}

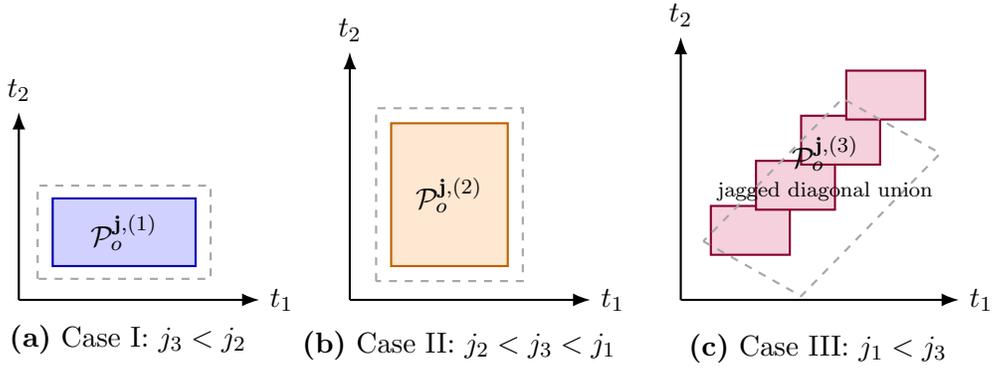
\begin{figure}[htbp]
\centering
\begin{tikzpicture}[
    x=1cm,y=1cm,
    axis/.style={-Latex,thick},
    tube/.style={dashed,gray!70,thick},
    rawA/.style={fill=blue!18,draw=blue!70!black,thick},
    rawB/.style={fill=orange!18,draw=orange!75!black,thick},
    rawC/.style={fill=purple!18,draw=purple!70!black,thick}
]

\begin{scope}[shift={(0,0)}]
    \draw[axis] (0,0) -- (3.2,0) node[right] {$t_1$};
    \draw[axis] (0,0) -- (0,2.5) node[above] {$t_2$};

    \fill[rawA] (0.45,0.45) rectangle (2.35,1.35);
    \draw[tube] (0.25,0.28) rectangle (2.55,1.52);

    \node at (1.40,0.90) {$\mathcal P_o^{\mathbf j,(1)}$};
    \node[align=center] at (1.45,-0.70)
    {\textbf{(a)} Case I: $j_3<j_2$\\ 
    };
\end{scope}

\begin{scope}[shift={(4.4,0)}]
    \draw[axis] (0,0) -- (3.2,0) node[right] {$t_1$};
    \draw[axis] (0,0) -- (0,3.3) node[above] {$t_2$};

    \fill[rawB] (0.55,0.45) rectangle (2.10,2.35);
    \draw[tube] (0.35,0.25) rectangle (2.30,2.55);

    \node at (1.32,1.40) {$\mathcal P_o^{\mathbf j,(2)}$};
    \node[align=center] at (1.45,-0.78)
    {\textbf{(b)} Case II: $j_2<j_3<j_1$\\ 
    };
\end{scope}

\begin{scope}[shift={(8.8,0)}]
    \draw[axis] (0,0) -- (3.7,0) node[right] {$t_1$};
    \draw[axis] (0,0) -- (0,3.5) node[above] {$t_2$};

    \foreach \a in {0.15,0.75,1.35,1.95} {
        \fill[rawC] (\a+0.25,\a+0.45) rectangle (\a+1.30,\a+1.10);
    }

    \draw[tube] (0.30,0.78) -- (1.60,0.05) -- (3.45,1.95) -- (2.15,2.68) -- cycle;

    \node[align=center] at (1.92,1.75) {$\mathcal P_o^{\mathbf j,(3)}$\\[-2pt]\scriptsize jagged diagonal union};
    \node[align=center] at (1.82,-0.82)
    {\textbf{(c)} Case III: $j_1<j_3$\\ 
    };
\end{scope}

\end{tikzpicture}
\caption{Central $(t_1,t_2)$-slices of the raw projected basic shards. The spatial base $\Box_{\mathbf j}^{+}$ is suppressed. In the first two regimes the raw projected shard is product-like, while in the third regime it is a diagonal union comparable to a slanted quotient tube. The dashed outlines indicate the corresponding tube geometry only schematically.}
\label{fig:raw_projected_shards}
\end{figure}

\subsection{Analytic dyadic shards}

We now build the dyadic objects that will be used in the martingale and Haar
analysis.  They are modeled on three exact product geometries.

\medskip
\noindent
\textbf{Parabolic dyadic blocks.}
For $\mu\in\{1,2,3\}$, let $\mathscr D_{j_\mu}^{(\mu)}$ denote the standard
dyadic cubes in $\mathbb C^{n_\mu}\cong \mathbb R^{2n_\mu}$ of side length
$2^{j_\mu}$, and let $\mathscr I_{j_\mu}^{(\mu)}$ denote the standard dyadic
intervals in $\mathbb R$ of length $2^{2j_\mu+\kappa}$.  A
\emph{parabolic dyadic block of scale $j_\mu$} in
$\mathbb C^{n_\mu}\times \mathbb R$ is a set of the form
\[
Q_\mu = I_\mu\times J_\mu,
\qquad
I_\mu\in \mathscr D_{j_\mu}^{(\mu)},
\quad
J_\mu\in \mathscr I_{j_\mu}^{(\mu)}.
\]
We denote the family of all such blocks by $\mathscr Q_{j_\mu}^{(\mu)}$.

\medskip
\noindent
\textbf{Three rectified product systems.}
For a scale $\mathbf j=(j_1,j_2,j_3)\in \mathbb Z^3$, define
\begin{align}
\mathfrak Q_{\mathbf j}^{(1)}
&:=
\Big\{
Q_1\times Q_2\times I_3:
Q_1\in \mathscr Q_{j_1}^{(1)},\,
Q_2\in \mathscr Q_{j_2}^{(2)},\,
I_3\in \mathscr D_{j_3}^{(3)}
\Big\},
\label{eq:rectified_family1}
\\
\mathfrak Q_{\mathbf j}^{(2)}
&:=
\Big\{
Q_1\times I_2\times Q_3:
Q_1\in \mathscr Q_{j_1}^{(1)},\,
I_2\in \mathscr D_{j_2}^{(2)},\,
Q_3\in \mathscr Q_{j_3}^{(3)}
\Big\}.
\label{eq:rectified_family2}
\end{align}
We also define the central shear
\begin{align}
\Theta(\mathbf z,t_1,t_2)
:=
(\mathbf z,t_1-t_2,t_2).
\label{eq:central_shear}
\end{align}
It is a linear measure-preserving bijection of $\mathscr N$.

The first two analytic dyadic systems are simply the product families
\[
\mathfrak R_{\mathbf j}^{(1)}:=\mathfrak Q_{\mathbf j}^{(1)},
\qquad
\mathfrak R_{\mathbf j}^{(2)}:=\mathfrak Q_{\mathbf j}^{(2)},
\]
while the third one is defined by
\begin{align}
\mathfrak R_{\mathbf j}^{(3)}
:=
\Theta^{-1}\big(\mathfrak Q_{\mathbf j}^{(2)}\big)
=
\big\{
\Theta^{-1}(Q): Q\in \mathfrak Q_{\mathbf j}^{(2)}
\big\}.
\label{eq:analytic_family3}
\end{align}

Thus the analytic dyadic shards are:
\begin{itemize}
    \item[(1)] a pure product family of type
    $(\mathbb C^{n_1}\times\mathbb R)\times
      (\mathbb C^{n_2}\times\mathbb R)\times
      \mathbb C^{n_3}$;
    \item[(2)] a mixed product family of type
    $(\mathbb C^{n_1}\times\mathbb R)\times
      \mathbb C^{n_2}\times
      (\mathbb C^{n_3}\times\mathbb R)$;
    \item[(3)] the inverse shear of the mixed product family in (2).
\end{itemize}

\begin{defn}[Admissible quotient dyadic structure]
\label{def:admissible_quotient_dyadic}~\\
The \emph{admissible quotient dyadic structure} on $\mathscr N$ is the union
\[
\mathfrak R
:=
\mathfrak R^{(1)}\cup \mathfrak R^{(2)}\cup \mathfrak R^{(3)},
\qquad
\mathfrak R^{(k)}
:=
\bigcup_{\mathbf j\in\mathbb Z^3}\mathfrak R_{\mathbf j}^{(k)}.
\]
We call the sets in $\mathfrak R_{\mathbf j}^{(k)}$
\emph{analytic dyadic shards of type $k$ and scale $\mathbf j$}.
\end{defn}

\begin{lem}\label{lem:analytic-nested}
For each fixed $k\in\{1,2,3\}$, the family $\mathfrak R^{(k)}$ is nested across scales in the following sense: if $\mathbf j,\mathbf \ell\in\mathbb Z^3$ satisfy
\[
j_\mu\le \ell_\mu,\qquad \mu=1,2,3,
\]
then for every $R\in \mathfrak R_{\mathbf \ell}^{(k)}$ there exists a unique $R'\in \mathfrak R_{\mathbf j}^{(k)}$ such that
\[
R\subseteq R'.
\]
\end{lem}

\begin{proof}
For $k=1,2$, this is exactly the standard nesting property of the underlying product dyadic cubes and dyadic intervals in the rectified coordinates.

For $k=3$, the family $\mathfrak R^{(3)}$ is defined by
\[
\mathfrak R^{(3)}=\Theta^{-1}(\mathfrak Q^{(2)}),
\]
where $\mathfrak Q^{(2)}$ is a standard product dyadic family and hence nested across scales. Since the bijection $\Theta^{-1}$ preserves set inclusion, the same nesting property holds for $\mathfrak R^{(3)}$.
\end{proof}

\begin{prop}[Basic properties of the analytic dyadic shards]
\label{prop:analytic_shards}~\\
The admissible quotient dyadic structure satisfies the following properties.

\begin{itemize}
    \item[(i)] For each $k\in\{1,2,3\}$ and each
    $\mathbf j\in\mathbb Z^3$, the family
    $\mathfrak R_{\mathbf j}^{(k)}$ is a partition of $\mathscr N$.

    \item[(ii)] For each fixed $k$, the collection
    $\mathfrak R^{(k)}$ is nested across scales.

    \item[(iii)] There exists an integer $\sigma\ge 1$ such that the following
    tube comparisons hold uniformly:
    \begin{align}
    T(\mathbf g\zeta_{\mathbf j},2^{\mathbf j-\sigma})
    \subset R \subset
    T(\mathbf g,2^{\mathbf j+\sigma}),
    \label{eq:analytic_tube_comparison}
    \end{align}
    whenever
    \begin{itemize}
        \item[$\bullet$] $R\in \mathfrak R_{\mathbf j}^{(1)}$ and $j_3<j_2$;
        \item[$\bullet$] $R\in \mathfrak R_{\mathbf j}^{(2)}$ and $j_2<j_3<j_1$;
        \item[$\bullet$] $R\in \mathfrak R_{\mathbf j}^{(3)}$ and $j_1<j_3$.
    \end{itemize}

    \item[(iv)] In each of the above three regimes, every analytic dyadic shard
    is uniformly comparable to a raw projected shard of the same scale.
\end{itemize}
\end{prop}

\begin{proof}
The partition statements follow directly from the corresponding partition properties of the underlying dyadic cubes and intervals in the rectified coordinates, together with the definition of $\mathfrak R_{\mathbf j}^{(3)}$ as the pullback of $\mathfrak Q_{\mathbf j}^{(2)}$ under $\Theta^{-1}$. The nesting across scales is exactly Lemma~\ref{lem:analytic-nested}.

For the tube comparison, note first that a block
$R\in \mathfrak R_{\mathbf j}^{(1)}$ is a product set
\[
I_1\times J_1\times I_2\times J_2\times I_3
\]
with
\[
\ell(I_\mu)=2^{j_\mu},\qquad
\ell(J_1)=2^{2j_1+\kappa},\qquad
\ell(J_2)=2^{2j_2+\kappa}.
\]
Hence $R$ is comparable to
$\Box_{\mathbf j}^+\times ([-2^{2j_1+\kappa},2^{2j_1+\kappa})\times
[-2^{2j_2+\kappa},2^{2j_2+\kappa}))$, which in turn is comparable to the
quotient tube in the regime $j_3<j_2$.

Similarly, a block $R\in \mathfrak R_{\mathbf j}^{(2)}$ is comparable to
\[
\Box_{\mathbf j}^+\times
([-2^{2j_1+\kappa},2^{2j_1+\kappa})\times
 [-2^{2j_3+\kappa},2^{2j_3+\kappa})),
\]
which is the product model corresponding to the intermediate regime
$j_2<j_3<j_1$.

Finally, if $R\in \mathfrak R_{\mathbf j}^{(3)}$, then
$\Theta(R)\in \mathfrak Q_{\mathbf j}^{(2)}$ is such a product block, and hence
\[
R=\Theta^{-1}(\Theta(R))
\]
is comparable to
\[
\Box_{\mathbf j}^+\times P_{c2^{2j_1},\,C2^{2j_3}}
\]
for uniform constants $c,C>0$,
 where \(P_{a,b}\subseteq \mathbb R^2\) denotes a strip-type slanted central region whose image in the \((u,v)=(t_1-t_2,t_2)\)-coordinates is comparable to a rectangle of side lengths \(a\) and \(b\).  This is exactly the twisted central geometry of
the quotient tube in the regime $j_1<j_3$.

In each regime, the raw projected shard and the analytic dyadic shard are both
uniformly comparable to the same quotient tube $T(\mathbf g,2^{\mathbf j})$.
Hence they are uniformly comparable to one another as well.
\end{proof}

%

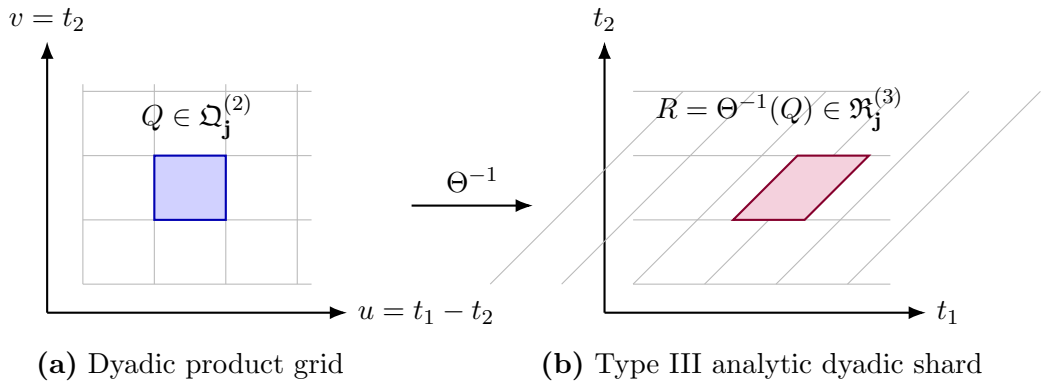
\begin{figure}[htbp]
\centering
\begin{tikzpicture}[
    x=0.95cm,y=0.95cm,
    axis/.style={-Latex,thick},
    gridline/.style={gray!55,thin},
    rectblock/.style={fill=blue!18,draw=blue!70!black,thick},
    shearblock/.style={fill=purple!18,draw=purple!70!black,thick}
]

\begin{scope}[shift={(0,0)}]
    \draw[axis] (0,0) -- (4.2,0) node[right] {$u=t_1-t_2$};
    \draw[axis] (0,0) -- (0,3.8) node[above] {$v=t_2$};

    \foreach \x in {0.5,1.5,2.5,3.5} {
        \draw[gridline] (\x,0.4) -- (\x,3.2);
    }
    \foreach \y in {0.4,1.3,2.2,3.1} {
        \draw[gridline] (0.5,\y) -- (3.7,\y);
    }

    \fill[rectblock] (1.5,1.3) rectangle (2.5,2.2);

    \node at (2.1,2.75) {$Q\in \mathfrak Q_{\mathbf j}^{(2)}$};
    \node[align=center] at (2.0,-0.75)
    {\textbf{(a)} Dyadic product grid};
\end{scope}

\begin{scope}[shift={(5.1,1.5)}]
    \draw[-Latex,thick] (0,0) -- (1.7,0);
    \node[above] at (0.85,0.05) {$\Theta^{-1}$};
\end{scope}

\begin{scope}[shift={(7.8,0)}]
    \draw[axis] (0,0) -- (4.5,0) node[right] {$t_1$};
    \draw[axis] (0,0) -- (0,3.8) node[above] {$t_2$};

    \foreach \y in {0.4,1.3,2.2,3.1} {
        \draw[gridline] (0.4,\y) -- (4.0,\y);
    }

    \foreach \c in {-2.0,-1.0,0.0,1.0,2.0,3.0} {
        \draw[gridline] ({0.4+\c},0.4) -- ({3.1+\c},3.1);
    }

    \fill[shearblock] (1.8,1.3) -- (2.8,1.3) -- (3.7,2.2) -- (2.7,2.2) -- cycle;

    \node at (2.45,2.85) {$R=\Theta^{-1}(Q)\in \mathfrak R_{\mathbf j}^{(3)}$};
    \node[align=center] at (2.2,-0.75)
    {\textbf{(b)} Type ${\rm III}$ analytic dyadic shard};
\end{scope}

\end{tikzpicture}
\caption{The analytic Type ${\rm III}$ dyadic shards are obtained as inverse images of a genuine dyadic product grid under the measure-preserving shear $\Theta(\mathbf z,t_1,t_2)=(\mathbf z,t_1-t_2,t_2)$. Thus the rectified side is Cartesian, while the physical side is a slanted dyadic tiling by analytic shards.}
\label{fig:analytic_dyadic_shards}
\end{figure}

For each $k\in\{1,2,3\}$ and $\mathbf j\in\mathbb Z^3$, let
\[
\mathscr F_{\mathbf j}^{(k)}
:=
\sigma\big(\mathfrak R_{\mathbf j}^{(k)}\big)
\]
be the corresponding dyadic $\sigma$-algebra, and let
$E_{\mathbf j}^{(k)}$ denote the associated conditional expectation.

\begin{prop}[Dyadic filtrations on $\mathscr N$]
\label{prop:quotient_filtrations}~\\
For each $k\in\{1,2,3\}$, the family
$\{\mathscr F_{\mathbf j}^{(k)}\}_{\mathbf j\in\mathbb Z^3}$
forms a tri-parameter dyadic family of $\sigma$-algebras, increasing in each coordinate separately.
  In particular, if
$R_{\mathbf j}^{(k)}(\mathbf g)$ denotes the unique analytic shard of type $k$
and scale $\mathbf j$ containing $\mathbf g$, then
\begin{align*}
E_{\mathbf j}^{(k)}f(\mathbf g)
=
\frac{1}{|R_{\mathbf j}^{(k)}(\mathbf g)|}
\int_{R_{\mathbf j}^{(k)}(\mathbf g)} f(\mathbf h)\,d\mathbf h.
\end{align*}
\end{prop}

\begin{proof}
For $k=1,2$, the filtrations are standard product dyadic filtrations in the rectified coordinates. For $k=3$, the filtration is obtained by pulling back the type $2$ rectified filtration under the measure-preserving shear $\Theta^{-1}$. The averaging formula follows because the atoms of $\mathscr F_{\mathbf j}^{(k)}$ are precisely the analytic dyadic shards of type $k$ and scale $\mathbf j$.
\end{proof}

\begin{remark}
The point of Section \ref{sec:quotient_shards} is the following.
The raw projected shards coming from the Heisenberg quotient construction
provide the correct geometric scale and the correct comparison with the quotient
tubes.  The analytic dyadic shards are the exact dyadic objects used in the
discrete analysis.  They coincide with the raw geometry up to uniform
comparability, but they are now nested and exactly rectifiable, which is what
Section \ref{sec:nilpotent_haar} requires.
\end{remark}

\section{Twisted nilpotent Haar systems and tight frame representation}
\label{sec:nilpotent_haar}
\setcounter{equation}{0}

We now build the discrete Haar theory on the analytic dyadic shards introduced
in Section \ref{sec:quotient_shards}. 
 The analytic dyadic shards serve as the dyadic objects in the discrete theory on $\mathscr N$. They form exact dyadic systems for Haar expansions and martingale decompositions, and at the same time they remain uniformly comparable to the raw projected shards, hence to the quotient tube geometry coming from the continuous theory.

For $k=1,2$, the analytic dyadic systems are exact product systems in rectified coordinates. For $k=3$, the dyadic system is obtained as the inverse image of the type ${\rm II}$ rectified system under the measure-preserving shear
\[
\Theta(\mathbf z,t_1,t_2)=(\mathbf z,t_1-t_2,t_2).
\]
Thus the Haar theory on $\mathscr N$ is reduced to standard product Haar expansions in rectified coordinates together with the pullback induced by $\Theta$.

\begin{remark}
The raw projected shards are not used directly as Haar supports, because they do not form an exact multiscale dyadic system. Their role is geometric: they indicate the relevant quotient regimes and tube scales. The analytic dyadic shards are the exact dyadic sets on which the Haar and martingale theory is carried out, while remaining comparable to the raw projected geometry.
\end{remark}

\subsection{Rectified Haar systems}

For a dyadic cube $I\subset \mathbb C^{n_\mu}\cong \mathbb R^{2n_\mu}$, let
\[
\{h_I^{\epsilon_\mu}\}_{\epsilon_\mu\in \mathcal E_\mu},
\qquad
\mathcal E_\mu:=\{1,\dots,2^{2n_\mu}-1\},
\]
be the standard $L^2$-normalized Haar basis on $I$.

For a parabolic dyadic block
\[
Q_\mu = I_\mu\times J_\mu \subset \mathbb C^{n_\mu}\times\mathbb R,
\]
let
\[
\{h_{Q_\mu}^{\varepsilon_\mu}\}_{\varepsilon_\mu\in \mathcal E_\mu^\sharp},
\qquad
\mathcal E_\mu^\sharp:=\{1,\dots,2^{2n_\mu+1}-1\},
\]
be the standard tensor-product Haar basis on the product block $Q_\mu$.

\medskip
\noindent
\textbf{Type 1 rectified basis.}
If
\[
Q = Q_1\times Q_2\times I_3 \in \mathfrak Q_{\mathbf j}^{(1)},
\]
with $Q_1\in \mathscr Q_{j_1}^{(1)}$, $Q_2\in \mathscr Q_{j_2}^{(2)}$, and
$I_3\in \mathscr D_{j_3}^{(3)}$, then for
\[
\vec\varepsilon=(\varepsilon_1,\varepsilon_2,\varepsilon_3)
\in
\mathcal E_1^\sharp\times \mathcal E_2^\sharp\times \mathcal E_3
\]
we define
\begin{align}
h_Q^{\vec\varepsilon,1}
(\mathbf z_1,t_1,\mathbf z_2,t_2,\mathbf z_3)
:=
h_{Q_1}^{\varepsilon_1}(\mathbf z_1,t_1)\,
h_{Q_2}^{\varepsilon_2}(\mathbf z_2,t_2)\,
h_{I_3}^{\varepsilon_3}(\mathbf z_3).
\label{eq:type1_haar}
\end{align}

\medskip
\noindent
\textbf{Type 2 rectified basis.}
If
\[
Q = Q_1\times I_2\times Q_3 \in \mathfrak Q_{\mathbf j}^{(2)},
\]
with $Q_1\in \mathscr Q_{j_1}^{(1)}$, $I_2\in \mathscr D_{j_2}^{(2)}$, and
$Q_3\in \mathscr Q_{j_3}^{(3)}$, then for
\[
\vec\varepsilon=(\varepsilon_1,\varepsilon_2,\varepsilon_3)
\in
\mathcal E_1^\sharp\times \mathcal E_2\times \mathcal E_3^\sharp
\]
we define
\begin{align}
h_Q^{\vec\varepsilon,2}
(\mathbf z_1,t_1,\mathbf z_2,\mathbf z_3,t_2)
:=
h_{Q_1}^{\varepsilon_1}(\mathbf z_1,t_1)\,
h_{I_2}^{\varepsilon_2}(\mathbf z_2)\,
h_{Q_3}^{\varepsilon_3}(\mathbf z_3,t_2).
\label{eq:type2_haar}
\end{align}

The collections
\[
\mathcal H_{\mathrm{std}}^{(1)}
:=
\left\{h_Q^{\vec\varepsilon,1}:Q\in \mathfrak Q^{(1)},\,\vec\varepsilon\right\},
\qquad
\mathcal H_{\mathrm{std}}^{(2)}
:=
\left\{h_Q^{\vec\varepsilon,2}:Q\in \mathfrak Q^{(2)},\,\vec\varepsilon\right\}
\]
are complete orthonormal bases for $L^2(\mathscr N)$, because they are standard
tensor-product Haar bases on the product filtrations
$\mathfrak Q^{(1)}$ and $\mathfrak Q^{(2)}$   in the rectified coordinates.

\subsection{Twisted Haar systems on analytic shards}

Define
\[
\Theta_1:=\mathrm{Id},
\qquad
\Theta_2:=\mathrm{Id},
\qquad
\Theta_3:=\Theta,
\]
with $\Theta$ given by \eqref{eq:central_shear}, and let
\[
U_kf:=f\circ \Theta_k,
\qquad
k=1,2,3.
\]
Then $U_1$ and $U_2$ are the identity, while $U_3$ is unitary on
$L^2(\mathscr N)$ and an isometry on every $L^p(\mathscr N)$.

For $k=1,2$, the analytic dyadic shards are already product blocks, so we set
\[
\mathcal B^{(1)}:=\mathcal H_{\mathrm{std}}^{(1)},
\qquad
\mathcal B^{(2)}:=\mathcal H_{\mathrm{std}}^{(2)}.
\]
For the third family, if $R\in \mathfrak R^{(3)}$ and
$Q=\Theta(R)\in \mathfrak Q^{(2)}$, then for
$\vec\varepsilon\in \mathcal E_1^\sharp\times \mathcal E_2\times
\mathcal E_3^\sharp$ we define
\begin{align}
h_R^{\vec\varepsilon,3}(\mathbf g)
:=
h_Q^{\vec\varepsilon,2}(\Theta(\mathbf g))
=
U_3\big(h_Q^{\vec\varepsilon,2}\big)(\mathbf g).
\label{eq:type3_haar}
\end{align}
We set
\[
\mathcal B^{(3)}
:=
\left\{h_R^{\vec\varepsilon,3}:R\in \mathfrak R^{(3)},\ \vec\varepsilon\right\}.
\]

\begin{lem}\label{lem:nilpotent-unitary}
Let \(U_3f:=f\circ\Theta\), where \(\Theta(\mathbf z,t_1,t_2)=(\mathbf z,t_1-t_2,t_2)\). Then the pullback $U_3$ is unitary on $L^2(\mathscr N)$ and an isometry on $L^p(\mathscr N)$ for every $1\le p\le \infty$. Moreover, if $E_{\mathbf j}^{(2)}$ denotes the conditional expectation associated with the type $2$ rectified dyadic structure, then
\[
E_{\mathbf j}^{(3)}=U_3 E_{\mathbf j}^{(2)} U_3^*,
\qquad
\Delta_{\mathbf j}^{(3)}=U_3 \Delta_{\mathbf j}^{(2)} U_3^*.
\]
\end{lem}

\begin{proof}
Since $\Theta$ is a linear measure-preserving bijection, the pullback $U_3$ preserves all $L^p$ norms. The identities for conditional expectations and martingale differences follow from the fact that $\mathfrak R^{(3)}$ is the pullback of $\mathfrak Q^{(2)}$ under $\Theta^{-1}$.
\end{proof}

\begin{prop}[Properties of the twisted nilpotent Haar systems]
\label{prop:twisted_nilpotent_haar}~\\
The three collections $\mathcal B^{(1)}$, $\mathcal B^{(2)}$, and
$\mathcal B^{(3)}$ satisfy the following properties.

\begin{itemize}
    \item[(i)] Each $\mathcal B^{(k)}$ is a complete orthonormal basis of
    $L^2(\mathscr N)$.

    \item[(ii)] Every basis element is supported on its corresponding analytic
    dyadic shard.

    \item[(iii)] If $\Delta_{\mathbf j}^{(1)}$ and $\Delta_{\mathbf j}^{(2)}$
    denote the standard tri-parameter martingale difference operators
    associated with the filtrations generated by
    $\mathfrak Q_{\mathbf j}^{(1)}$ and $\mathfrak Q_{\mathbf j}^{(2)}$,
    respectively, and if
    \[
    \Delta_{\mathbf j}^{(3)}
    :=
    U_3\,\Delta_{\mathbf j}^{(2)}\,U_3^*,
    \]
    then for each $k\in\{1,2,3\}$,
    \begin{align}
    \Delta_{\mathbf j}^{(k)}f
    =
    \sum_{R\in \mathfrak R_{\mathbf j}^{(k)}}
    \ \sum_{\vec\varepsilon}
    \langle f,h_R^{\vec\varepsilon,k}\rangle\,
    h_R^{\vec\varepsilon,k},
    \label{eq:nilpotent_martingale_projection}
    \end{align}
    with convergence in $L^2(\mathscr N)$.
\end{itemize}
\end{prop}

\begin{proof}

We first treat the type ${\rm I}$ and type ${\rm II}$ systems. 

By construction, the families
$\mathfrak R^{(1)}=\mathfrak Q^{(1)}$ and $\mathfrak R^{(2)}=\mathfrak Q^{(2)}$
are exact product dyadic systems in rectified coordinates. Hence
$\mathcal B^{(1)}=\mathcal H_{\mathrm{std}}^{(1)}$ and
$\mathcal B^{(2)}=\mathcal H_{\mathrm{std}}^{(2)}$ are precisely the standard tensor-product Haar bases on the corresponding product filtrations. In particular, both $\mathcal B^{(1)}$ and $\mathcal B^{(2)}$ are complete orthonormal bases of $L^2(\mathscr N)$, every basis element is supported on its corresponding analytic dyadic shard, and the martingale difference expansions in \eqref{eq:nilpotent_martingale_projection} are exactly the standard product Haar expansions for $k=1,2$.

We now turn to the type ${\rm III}$ system. Recall that $\mathfrak R^{(3)}$ is the pullback of $\mathfrak Q^{(2)}$ under the measure-preserving bijection $\Theta^{-1}$, and that
\[
h_R^{\vec\varepsilon,3}
=
U_3\big(h_Q^{\vec\varepsilon,2}\big)
=
h_Q^{\vec\varepsilon,2}\circ \Theta,
\qquad Q=\Theta(R).
\]
Let $R,R'\in \mathfrak R^{(3)}$, and let $Q=\Theta(R)$ and $Q'=\Theta(R')$ belong to $\mathfrak Q^{(2)}$. Since $\Theta$ is measure-preserving, we have
\begin{align*}
\langle h_R^{\vec\varepsilon,3}, h_{R'}^{\vec{\varepsilon'},3}\rangle
&=
\int_{\mathscr N}
h_Q^{\vec\varepsilon,2}(\Theta(\mathbf g))
\overline{h_{Q'}^{\vec{\varepsilon'},2}(\Theta(\mathbf g))}
\,d\mathbf g \\
&=
\int_{\mathscr N}
h_Q^{\vec\varepsilon,2}(\mathbf y)
\overline{h_{Q'}^{\vec{\varepsilon'},2}(\mathbf y)}
\,d\mathbf y \\
&=
\delta_{Q,Q'}\delta_{\vec\varepsilon,\vec{\varepsilon'}}.
\end{align*}
Thus $\mathcal B^{(3)}$ is orthonormal. Since it is the unitary image under $U_3$ of the complete orthonormal basis $\mathcal B^{(2)}$, it is also complete in $L^2(\mathscr N)$. Moreover,
\[
\mathrm{supp}\big(h_R^{\vec\varepsilon,3}\big)
=
\Theta^{-1}\big(\mathrm{supp}(h_Q^{\vec\varepsilon,2})\big)
=
\Theta^{-1}(Q)
=
R,
\]
so each type ${\rm III}$ Haar function is supported on its corresponding analytic dyadic shard.

It remains to verify the martingale projection identity for $k=3$. By Lemma~\ref{lem:nilpotent-unitary},
\[
\Delta_{\mathbf j}^{(3)}=U_3\,\Delta_{\mathbf j}^{(2)}\,U_3^*.
\]
Applying the standard product Haar expansion for $\Delta_{\mathbf j}^{(2)}$ and then conjugating by $U_3$, we obtain
\begin{align*}
\Delta_{\mathbf j}^{(3)}f
&=
U_3\Bigg(
\sum_{Q\in \mathfrak Q_{\mathbf j}^{(2)}}
\sum_{\vec\varepsilon}
\langle U_3^*f,h_Q^{\vec\varepsilon,2}\rangle\,
h_Q^{\vec\varepsilon,2}
\Bigg) \\
&=
\sum_{Q\in \mathfrak Q_{\mathbf j}^{(2)}}
\sum_{\vec\varepsilon}
\langle f,U_3h_Q^{\vec\varepsilon,2}\rangle\,
U_3h_Q^{\vec\varepsilon,2}.
\end{align*}
Since $U_3h_Q^{\vec\varepsilon,2}=h_R^{\vec\varepsilon,3}$ with $R=\Theta^{-1}(Q)$, we have
\begin{align*}
\Delta_{\mathbf j}^{(3)}f
=
\sum_{R\in \mathfrak R_{\mathbf j}^{(3)}}
\sum_{\vec\varepsilon}
\langle f,h_R^{\vec\varepsilon,3}\rangle\,
h_R^{\vec\varepsilon,3}.
\end{align*}
Combining this with the established product expansions for $k=1,2$ proves \eqref{eq:nilpotent_martingale_projection} for all $k\in\{1,2,3\}$, and completes the proof.
\end{proof}

Observe that the three families $\mathcal B^{(1)}$, $\mathcal B^{(2)}$, and $\mathcal B^{(3)}$ are three distinct complete orthonormal bases of $L^2(\mathscr N)$. Accordingly, their union is not a Parseval frame but a tight frame with frame bound $3$.

\begin{thm}[Twisted nilpotent Haar tight frame]
\label{thm:nilpotent_tight_frame}~\\
The union
$\mathcal B^{(1)}\cup \mathcal B^{(2)}\cup \mathcal B^{(3)}$
is a tight frame for $L^2(\mathscr N)$ with frame bound $3$.  Equivalently,
for every $f\in L^2(\mathscr N)$,
\begin{align}
\sum_{k=1}^3
\ \sum_{R\in \mathfrak R^{(k)}}
\ \sum_{\vec\varepsilon}
\big|\langle f,h_R^{\vec\varepsilon,k}\rangle\big|^2
=
3\|f\|_{L^2(\mathscr N)}^2,
\label{eq:nilpotent_frame_energy}
\end{align}
and
\begin{align}
f
=
\frac{1}{3}
\sum_{k=1}^3
\ \sum_{R\in \mathfrak R^{(k)}}
\ \sum_{\vec\varepsilon}
\langle f,h_R^{\vec\varepsilon,k}\rangle\,h_R^{\vec\varepsilon,k}
\label{eq:nilpotent_frame_reconstruction}
\end{align}
with unconditional convergence in $L^2(\mathscr N)$.
\end{thm}

\begin{proof}
By Proposition \ref{prop:twisted_nilpotent_haar}, each $\mathcal B^{(k)}$ is a
complete orthonormal basis of $L^2(\mathscr N)$.  Therefore Parseval's identity
gives
\[
\sum_{R\in \mathfrak R^{(k)}}
\ \sum_{\vec\varepsilon}
\big|\langle f,h_R^{\vec\varepsilon,k}\rangle\big|^2
=
\|f\|_{L^2(\mathscr N)}^2
\]
for each fixed $k$.  Summing over $k=1,2,3$ yields
\eqref{eq:nilpotent_frame_energy}, and the reconstruction formula
\eqref{eq:nilpotent_frame_reconstruction} is the standard formula for a tight
frame with frame bound $3$.
\end{proof}

\subsection{Twisted dyadic square functions}

For $k=1,2,3$, define the dyadic square function by
\begin{align}
S_d^{(k)}(f)(\mathbf g)
:=
\left(
\sum_{\mathbf j\in\mathbb Z^3}
\big|\Delta_{\mathbf j}^{(k)}f(\mathbf g)\big|^2
\right)^{1/2}.
\label{eq:nilpotent_square_function}
\end{align}
Equivalently, by \eqref{eq:nilpotent_martingale_projection},
\begin{align}
S_d^{(k)}(f)(\mathbf g)
=
\left(
\sum_{\mathbf j\in\mathbb Z^3}
\ \sum_{R\in \mathfrak R_{\mathbf j}^{(k)}}
\ \sum_{\vec\varepsilon}
\big|
\langle f,h_R^{\vec\varepsilon,k}\rangle\,
h_R^{\vec\varepsilon,k}(\mathbf g)
\big|^2
\right)^{1/2}.
\label{eq:nilpotent_square_function_coeff}
\end{align}

\begin{remark}
For $k=1,2$, the square functions $S_d^{(k)}$ are the standard dyadic square functions associated with exact product filtrations in rectified coordinates. Thus the usual product Littlewood--Paley theorem applies directly. The only genuinely twisted case is $k=3$, which is handled by conjugation through the measure-preserving shear $\Theta$.
\end{remark}

\begin{thm}[Dyadic Littlewood--Paley estimates on $\mathscr N$]
\label{thm:nilpotent_square_function_transfer}~\\
For every $1<p<\infty$ and every $k\in\{1,2,3\}$,
\begin{align}
\|S_d^{(k)}(f)\|_{L^p(\mathscr N)}
\simeq
\|f\|_{L^p(\mathscr N)}.
\label{eq:nilpotent_lp_equivalence}
\end{align}
The implicit constants depend only on $p$ and the dimensions.
\end{thm}

\begin{proof}
For $k=1$ and $k=2$, the systems $\mathfrak Q^{(1)}$ and $\mathfrak Q^{(2)}$
are standard product dyadic filtrations on the product spaces
\[
(\mathbb C^{n_1}\times \mathbb R)\times
(\mathbb C^{n_2}\times \mathbb R)\times
\mathbb C^{n_3}
\]
and
\[
(\mathbb C^{n_1}\times \mathbb R)\times
\mathbb C^{n_2}\times
(\mathbb C^{n_3}\times \mathbb R),
\]
respectively.  Hence the classical product Littlewood--Paley theorem yields
\[
\|S_d^{(1)}(f)\|_{L^p(\mathscr N)}
\simeq
\|f\|_{L^p(\mathscr N)},
\qquad
\|S_d^{(2)}(f)\|_{L^p(\mathscr N)}
\simeq
\|f\|_{L^p(\mathscr N)}.
\]

For $k=3$, the martingale difference operators satisfy
\[
\Delta_{\mathbf j}^{(3)}
=
U_3\,\Delta_{\mathbf j}^{(2)}\,U_3^*.
\]
Therefore
\begin{align*}
S_d^{(3)}(f)(\mathbf g)
&=
\left(
\sum_{\mathbf j\in\mathbb Z^3}
\big|
U_3\Delta_{\mathbf j}^{(2)}U_3^*f(\mathbf g)
\big|^2
\right)^{1/2}
\\
&=
\left(
\sum_{\mathbf j\in\mathbb Z^3}
\big|
\Delta_{\mathbf j}^{(2)}(U_3^*f)(\Theta(\mathbf g))
\big|^2
\right)^{1/2}
\\
&=
S_d^{(2)}(U_3^*f)(\Theta(\mathbf g))
\\
&=
U_3\big(S_d^{(2)}(U_3^*f)\big)(\mathbf g).
\end{align*}
Since $U_3$ and $U_3^*$ are isometries on every $L^p(\mathscr N)$,
\begin{align*}
\|S_d^{(3)}(f)\|_{L^p(\mathscr N)}
&=
\|S_d^{(2)}(U_3^*f)\|_{L^p(\mathscr N)}
\\
&\simeq
\|U_3^*f\|_{L^p(\mathscr N)}
=
\|f\|_{L^p(\mathscr N)}.
\end{align*}
This proves \eqref{eq:nilpotent_lp_equivalence}.
\end{proof}

\begin{remark}
The dyadic geometry of $\mathscr N$ splits into three model families. The first two are product-type systems, while the third is obtained as the inverse image of a product-type system under the central shear
\[
\Theta(\mathbf z,t_1,t_2)=(\mathbf z,t_1-t_2,t_2).
\]
These analytic dyadic families remain uniformly comparable to the corresponding quotient tube geometries, and they provide the dyadic setting for the Haar and Littlewood--Paley theory in this paper.
\end{remark}

\vskip 1cm

\noindent
{\bf Acknowledgments}:\   J. Li is supported by Australian Research Council (DP 220100285 and DP 260100485). C.W. Liang is supported by NSTC through grant 111-2115-M-002-010-MY5. B. Wick is supported
 in part by National Science Foundation
DMS awards (No. 2349868) and Australian Research Council (DP 220100285 and DP260201083).
L. Wu is supported by  NNSF of China (No. 12201002). 
Q. Wu is supported by NNSF of China (No. 12171221) and Taishan Scholars Program for Young Experts of Shandong Province (tsqn202507265).


\vskip 1cm


\begin{thebibliography}{99}

\bibitem{CF80}
{\sc S.-Y. A. Chang and R. Fefferman,}
\newblock A continuous version of duality of $H^1$ with BMO on the bidisc.
\newblock {\em Ann. of Math. 112}, 1 (1980), 179--201.

\bibitem{CF85}
{\sc S.-Y. A. Chang and R. Fefferman,}
\newblock Some recent developments in Fourier analysis and $H^p$-theory on product domains.
\newblock {\em Bull. Amer. Math. Soc. 12}, 1 (1985), 1--43.

\bibitem{CCLLO}
{\sc  P. Chen, M. Cowling, M.-Y.Lee, J. Li and A. Ottazzi, }
\newblock Flag Hardy space theory on Heisenberg groups and applications.
\newblock {\em  arXiv:2102.07371\/} (2021).

\bibitem{Feff-S}
{\sc R. Fefferman and E. M. Stein, }
\newblock Singular integrals on product spaces.
\newblock {\em Adv. in Math. 45}, 2 (1982), 117--143.

\bibitem{FLLWW}
{\sc Z. Fu, J. Li, C.W. Liang, W. Wang and Q. Wu,}
\newblock Twisted Multiparameter singular integrals---real variable methods and applications, I.
\newblock {\em arXiv:2603.26119} (2026).

\bibitem{HLLW}
{\sc Y. S. Han,  M. Y. Lee, J. Li and B. Wick,}
\newblock Maximal function, Littlewood--Paley theory, Riesz transform and atomic decomposition in the multi-parameter flag setting.
\newblock {\em Mem. Amer. Math. Soc. 279}, 1373 (2022).

\bibitem{J}
{\sc L. Journ\'e,}
\newblock A covering lemma for product spaces.
\newblock {\em Proc. Amer. Math. Soc. 96}, 4 (1986), 593--598.


\bibitem{KLPW} {\sc A. Kairema, J. Li, C. Pereyra and L.A. Ward,}  \newblock Haar bases on
quasi-metric measure spaces, and dyadic structure theorems for function spaces on
product spaces of homogeneous type. \newblock {\em  J. Funct. Anal. 271}   (2016), no. 7,  1793--1843.


\bibitem{LLWW}
{\sc J. Li, C.-W. Liang, C.J. Wen and Q. Wu,}
\newblock Maximal functions with twisted structures, distribution  inequality and applications.
\newblock {\em arXiv:2604.00345} (2026).

\bibitem{MRS}
{\sc D. M\"uller, F. Ricci and E.M. Stein,}
\newblock Marcinkiewicz multipliers and multi-parameter structure on Heisenberg (-type) groups I.
\newblock {\em Invent. Math. 119}, 1 (1995), 199--233.

\bibitem{NRS}
{\sc A. Nagel, F. Ricci and E. M. Stein,}
\newblock Singular integrals with flag kernels and analysis on quadratic CR manifolds.
\newblock {\em J. Funct. Anal. 181}, 1 (2001), 29--118.

\bibitem{NRSW1}
{\sc A. Nagel, F. Ricci, E. M. Stein and S. Wainger,}
\newblock Singular integrals with flag kernels on homogeneous groups: I.
\newblock {\em Rev. Mat. Iberoam. 28}, 3 (2012), 631--722.

\bibitem{NRSW2}
{\sc A. Nagel, F. Ricci, E. M. Stein and S. Wainger,}
\newblock Algebras of singular integral operators with kernels controlled by multiple norms.
\newblock {\em Mem. Amer. Math. Soc. 256}, 1230 (2018).

\bibitem{P}
{\sc J. Pipher,}
\newblock Journ\'e's covering lemma and its extension to higher dimensions.
\newblock {\em Duke Math. J. 53}, 3 (1986), 683--690.

\bibitem{St98}
{\sc E. M. Stein,}
\newblock Singular integrals: the roles of Calder\'on and Zygmund.
\newblock {\em Notices Amer. Math. Soc. 45}, 9 (1998), 1130--1140.

\bibitem{St}
{\sc R. Strichartz,}
\newblock Self-similarity on nilpotent Lie groups.
\newblock In {\em Generalized convex bodies and generalized envelopes}, Contemp. Math. 140. Amer. Math. Soc., Providence, RI, 1992, pp. 123--157.

\bibitem{Ty}
{\sc J. Tyson,}
\newblock Global conformal Assouad dimension in the Heisenberg group.
\newblock {\em Conf. Geom. Dyn. 12\/} (2008), 32--57.

\bibitem{MR4808333}
{\sc W. Wang and Q.~Y. Wu,} \newblock Flag-like singular integrals and associated Hardy spaces on a kind of nilpotent Lie groups of step two.
\newblock {\em J. Geom. Anal. 34\/}  (2024), no.~12, Paper No. 373, 51 pp.

\end{thebibliography}
\end{document}